\documentclass[a4paper,12pt,twoside]{amsart}
\usepackage[english]{babel}
\usepackage{amssymb, amsmath, eucal}
\usepackage[all]{xy}
\usepackage{mathrsfs}
\usepackage[dvips]{graphics}

\newtheorem{The}{Theorem}[section]
\newtheorem{Lem}[The]{Lemma}
\newtheorem{Prop}[The]{Proposition}
\newtheorem{Cor}[The]{Corollary}
\newtheorem{Rem}[The]{Remark}

\newtheorem{Def}[The]{Definition}

\newcommand{\C}{\mathbb{C}}
\newcommand{\R}{\mathbb{R}}
\newcommand{\N}{\mathbb{N}}

\newcommand{\U}{\underline{u}}

\begin{document}
\title[Parabolic complex Monge-Amp\`ere equations]{Weak solution of Parabolic complex Monge-Amp\`ere equation}

\setcounter{tocdepth}{1}

  \author{Do Hoang Son } 

\email{hoangson.do.vn@gmail.com }
 \date{\today}
 \begin{abstract}
   We study the equation $\dot{u}=\log\det (u_{\alpha\bar{\beta}})-Au+f(z,t)$ in  domains
   of $\C^n$. This equation has a close connection with the K\"ahler-Ricci flow. In this paper, we consider the case where the boundary condition is smooth and the initial condition is irregular.
 \end{abstract} 
 \maketitle
 
 \tableofcontents
 \newpage
 \section{Introduction}
 Let $\Omega$ be a bounded smooth strictly pseudoconvex domain of $\C^n$, i.e., there exists a smooth
 strictly plurisubharmonic function $\rho$ defined on a bounded neighbourhood of $\bar{\Omega}$ such that 
 $$
 \Omega=\{\rho<0\}.
 $$
  Let $A\geq 0,T>0$. We consider the equation
 \begin{equation}\label{KRF}
 \begin{cases}
 \begin{array}{ll}
 \dot{u}=\log\det (u_{\alpha\bar{\beta}})-Au+f(z,t)\;\;\;&\mbox{on}\;\Omega\times (0,T),\\
 u=\varphi&\mbox{on}\;\partial\Omega\times [0,T),\\
 u=u_0&\mbox{on}\;\bar{\Omega}\times\{ 0\},\\
 \end{array}
 \end{cases}
 \end{equation}
 where $\dot{u}=\frac{\partial u}{\partial t}$,
  $u_{\alpha\bar{\beta}}=\frac{\partial^2 u}{\partial z_{\alpha}\partial\bar{z}_{\beta}}$,
  $u_0$ is a plurisubharmonic function in a neighbourhood of $\bar{\Omega}$
  and $\varphi, f$ are smooth in $\bar{\Omega}\times [0,T]$.
  
  If $u_0$ is a smooth strictly plurisubharmonic function on $\bar{\Omega}$ and satisfies 
  the compatibility condition
  on $\partial\Omega\times\{0\}$
  $$
   \dot{\varphi}=\log\det (u_0)_{\alpha\bar{\beta}}-Au_0+f(z,0),
   $$
   then \eqref{KRF} admits a unique smooth solution \cite{HL10}. 
   After studying the case where $u_0$ is continuous or just bounded \cite{Do15}, 
we want to understand the situation when 
   $u_0$ is a more general plurisubharmonic functions, first with zero Lelong numbers, then
   in some special cases where positive Lelong numbers are involved.
   
   On compact K\"ahler manifolds, the corresponding problem was considered and solved \cite{GZ13, DL14}.
 By using approximations and a priori estimates, it was shown that the Parabolic complex Monge-Amp\`ere 
 equation admits a unique solution in a sense close to classical solution. We expect that a similar 
 situation obtains on the domain $\Omega$. 
 
 In order to study the situation with irregular initial data,
 we give a notion of ``weak solution" for \eqref{KRF}, consider the existence 
 of weak solutions, and ``describe" weak solutions in some particular cases.
 
  The function $u\in USC(\bar{\Omega}\times [0,T))$ (upper semicontinuous function) is called a 
  \emph{weak solution} of 
  \eqref{KRF} if there exist $u_m\in C^{\infty}(\bar{\Omega}\times [0,T))$ satisfying 
  \begin{equation}\label{KRF_weak}
   \begin{cases}
   \begin{array}{ll}
   u_m(.,t)\in SPSH(\Omega),\\
   \dot{u}_m=\log\det (u_m)_{\alpha\bar{\beta}}-Au_m+f(z,t)\;\;\;&\mbox{on}\;\Omega\times (0,T),\\
   u_m\searrow\varphi&\mbox{on}\;\partial\Omega\times [0,T),\\
   u_m\searrow u_0&\mbox{on}\;\bar{\Omega}\times\{ 0\},\\
   u=\lim\limits_{m\rightarrow \infty} u_m.
   \end{array}
   \end{cases}
   \end{equation}
   where $SPSH(\Omega)=\{\mbox{strictly plurisubharmonic functions on } \Omega\}$.
   
   For the convenience, we also denote by $k_A$ the functions
    \begin{equation}\label{kA.eq}
            \begin{cases}
            k_A(x,t)=x-2nt\;\;\mbox{ if } \;\; A=0,\\
            k_A(x,t)=-\frac{2n}{A}+\left(\frac{2n}{A}+x\right)e^{-At}\;\;\mbox{ if }\;\;  A>0,
            \end{cases}
   \end{equation}
   where $x>0$ and $0<t<T$, and by $\epsilon_A$ the value of $t$ such that $k_A(x,t)=0$:
     \begin{equation}\label{epsilonA.eq}
   \begin{cases}
       \epsilon_A(x)=\frac{x}{2n} \mbox{ if } A=0,\\
       \epsilon_A(x)=\frac{1}{A}(\log (Ax+2n)-\log (2n))  \mbox{ if } A>0.
    \end{cases}
   \end{equation}

   Our main results are the following:
   \begin{The}\label{main.existence}
   Let $A\geq 0, T>0$ and $\Omega$ be a bounded smooth strictly pseudoconvex domain of $\C^n$.
    Let $\varphi, f$ be smooth functions in $\bar{\Omega}\times [0,T]$ 
     and $u_0$ be a plurisubharmonic function in a neighbourhood of $\bar{\Omega}$
    such that $u_0(z)=\varphi(z,0)$ for any $z\in\partial\Omega$.
     Then \eqref{KRF} has a unique weak solution.
   \end{The}
   \begin{The}\label{main.zeroLelong}
   Suppose that the conditions of Theorem \ref{main.existence} are satisfied. Suppose also that
   $u_0$ has zero Lelong numbers, i.e., 
        $$\nu_{u_0}(a)=\lim\limits_{r\rightarrow 0}\dfrac{\sup_{|z-a|<r}u_0(z)}{\log r}=0,$$
  for any $z\in\Omega$. Then the weak solution $u$ satisfies
   \begin{itemize}
   \item[(a)]$u\in C^{\infty}(\bar{\Omega}\times (0,T))$.
   \item[(b)] $\dot{u}=\log\det u_{\alpha\bar{\beta}}-Au+f(z,t)$ in $\bar{\Omega}\times (0,T)$.
   \item[(c)]  $u(.,t)\stackrel{L^1}{\longrightarrow}u_0$ as $t\searrow 0;\;
        u|_{\partial\Omega\times [0,T)}=\varphi|_{\partial\Omega\times [0,T)}$.\\
   \end{itemize}
   \end{The}
   \begin{The}\label{main.nonsmooth}
      Suppose that the conditions of Theorem \ref{main.existence} are satisfied. If there is $a\in\Omega$ such that the Lelong number of $u_0$ at $a$ is positive
       then there is no $u\in C^{\infty}(\bar{\Omega}\times (0,T))$
      satisfying
      \begin{equation}\label{eq.KRF.main3}
       \begin{cases}
       \begin{array}{ll}
       u(.,t)\in SPSH(\bar{\Omega}),&\forall t\in (0,T),\\
       \dot{u}=\log\det (u_{\alpha\bar{\beta}})-Au+f(z,t)\;\;\;&\mbox{on}\;\bar{\Omega}\times (0,T),\\
       u=\varphi&\mbox{on}\;\partial\Omega\times [0,T),\\
       u(.,t)\stackrel{L^1}{\longrightarrow}u_0.\\
       \end{array}
       \end{cases}
       \end{equation}
      \end{The}
   \begin{The}\label{main.description}
   Suppose that the conditions of Theorem \ref{main.existence} are satisfied. Suppose also that 
    $$\sum\limits_{j=1}^l n_j\log |z-a_j|+C_0\geq u_0\geq \sum\limits_{j=1}^l N_j\log |z-a_j|-C_0,$$ 
    where  $l\in\N, a_j\in\Omega, N_j\geq n_j>0$.
   Then the weak solution $u$ satisfies
   \begin{itemize}
   \item[(a)] $u\in C^{\infty}(Q)$, where $Q=(\bar{\Omega}\times (0,T))\setminus
    (\cup(\{a_j\}\times (0,\epsilon_A(n_j)])$.
    \item[(b)] $u=-\infty$ on $\cup(\{a_j\}\times [0,\min\{T,\epsilon_A(N_j)\}))$.\\
    Moreover, if $n_j=N_j$ then 
    $$\nu_{u(.,t)}(a_j)= k(N_j,t),$$
    for any $0<t<\min\{T,\epsilon_A(N_j)\}$.
    \item[(c)] $\dot{u}=\log\det u_{\alpha\bar{\beta}}-Au+f(z,t)$ in $Q$.\\
    \item[(d)] $u(.,t)\stackrel{L^1}{\longrightarrow}u_0$ when $t\searrow 0;\;
     u|_{\partial\Omega\times [0,T)}=\varphi|_{\partial\Omega\times [0,T)}$.\\
     Moreover, $u(.,t)\rightarrow u_0$ ``uniformly in capacity", i.e.,
            if $\epsilon>0$ and $\bar{\Omega}\Subset\tilde{\Omega}$ then there exists 
            an open set $U_{\epsilon}$ such that
            \begin{center}
            $Cap_{\tilde{\Omega}}(\bar{\Omega}\setminus U_{\epsilon})\leq \epsilon$,\\
            $u(.,t)\rightrightarrows u_0$ on $\bar{\Omega}\cap U_{\epsilon}$ as $t\searrow 0$.
            \end{center}
   \end{itemize}
   \end{The}
   
 \section{Preliminaries}
 \subsection{Hou-Li theorem}
 \begin{flushleft}
 The Hou-Li theorem states
 that equation \eqref{KRF} has a unique solution when the conditions are good 
 enough. We state the precise problem to be studied:
 \end{flushleft}
 \begin{equation}
 \label{HLKRF}
 \begin{cases}
 \begin{array}{ll}
 \dot{u}=\log\det (u_{\alpha\bar{\beta}})+f(t,z,u)\;\;\;&\mbox{on}\;\Omega\times (0,T),\\
 u=\varphi&\mbox{on}\;\partial\Omega\times [0,T),\\
 u=u_0&\mbox{on}\;\bar{\Omega}\times\{ 0\}.\\
 \end{array}
 \end{cases} 
 \end{equation}
 
 We first need the notion of a subsolution to \eqref{HLKRF}.
 \begin{Def}
 \label{HLsubsol}
 A function $\underline{u}\in C^{\infty}(\bar{\Omega}\times [0,T))$
 is called a \emph{subsolution} of the equation \eqref{HLKRF} if and only if
 \begin{equation}
 \label{subsolu.houli}
 \begin{cases}
 \U(.,t) \mbox{is a strictly plurisubharmonic function,}\\
 \dot{\U}\leq \log\det (\U)_{\alpha\bar{\beta}}+f(t,z,\U),\\
 \U|_{\partial\Omega\times (0,T)}=\varphi|_{\partial\Omega\times (0,T)},\\
 \U(.,0)\leq u_0 .
 \end{cases} 
 \end{equation}
 
 \end{Def}
 \begin{The}\label{houli}
 Let $\Omega\subset\C^n$ be a bounded domain with smooth boundary. Let $T\in (0,\infty]$. Assume
 that
 \begin{itemize}
 \item $\varphi$ is a smooth function in $\bar{\Omega}\times [0,T)$.
 \item $f$ is a smooth function in $[0,T)\times\bar{\Omega}\times\R$ non increasing in the
 lastest variable.
 \item $u_0$ is a smooth strictly plurisubharmonic funtion in a neighborhood of $\Omega$.
 \item $u_0(z)=\varphi (z,0),\;\forall z\in\partial\Omega$.
 \item The compatibility condition is satisfied, i.e.
 $$
 \dot{\varphi}=\log\det (u_0)_{\alpha\bar{\beta}}+f(t,z,u_0),
 \;\;\forall (z,t)\in \partial\Omega\times \{ 0 \}.
 $$
 \item There exists a subsolution to the equation \eqref{HLKRF}.
 \end{itemize} 
 Then there exists a unique solution 
 $u\in C^{\infty}(\Omega\times (0,T))\cap C^{2;1}(\bar{\Omega}\times [0,T))$ of the equation
 \eqref{HLKRF}.
  \end{The}
 \begin{Rem}
 \begin{itemize}
 \item[(i)]There is a corresponding result in the case of a compact K\"ahler manifold \cite{Cao85}.
  On the compact  K\"ahler manifold $X$, we
 must assume that $0<T<T_{max}$, where $T_{max}$ depends on $X$. In the case of domain 
 $\Omega\subset\C^n$, we can assume that $T=+\infty$ if $\varphi, \U$ are defined on 
 $\bar{\Omega}\times [0,+\infty)$ and $f$ is defined on $[0,+\infty)\times\bar{\Omega}\times\R$.
 \item[(ii)] If $\Omega$ is a bounded smooth strictly pseudoconvex domain of $\C^n$ then 
 one can prove that a subsolution always exists on $\bar{\Omega}\times [0,T')$, for any 
 $0<T'<T$, and so Theorem \ref{houli} does not need the additional assumption of existence of a subsolution.
 \end{itemize}
 
 \end{Rem}
 \subsection{Maximum principle}
 \begin{flushleft}
 The following maximum principle is a basic tool to
 establish upper and lower bounds in the sequel (see \cite{BG13} and \cite{IS13} for the proof).
 \end{flushleft}
 \begin{The}\label{max prin}
 Let $\Omega$ be a bounded domain of $C^n$ and $T>0$. 
 Let $\{\omega_t\}_{0<t<T}$ be a continuous family of continuous 
   positive definite Hermitian forms on $\Omega$.  Denote by $\Delta_t$ the Laplacian with 
  respect to $\omega_t$:
  $$
  \Delta_t f=\dfrac{n\omega_t^{n-1}\wedge dd^cf}{\omega_t^n},\;\forall f\in C^{\infty}(\Omega).
  $$
    Suppose that
 $H \in C^{\infty}(\Omega \times (0,T)) \cap C(\bar{\Omega}\times [0,T))$ and satisfies\\
 \begin{center}
 $(\frac{\partial}{\partial t}-\Delta_t)H \leq 0 \:$ 
 or
 $\: \dot{H}_t \leq \log\dfrac{(\omega_t+dd^c H_t)^n}{\omega_t^n}$.
 \end{center}
 Then $\sup\limits_{\bar{\Omega}\times [0,T)} H = \sup\limits_{\partial_P(\Omega \times
 [0,T))}H$. Here we denote $\partial_P(\Omega\times (0,T))=\left(\partial\Omega\times (0,T) \right)
 \cup \left( \bar{\Omega}\times\{ 0\}\right)$.
 \end{The}
 \begin{Cor}(Comparison principle)\label{compa log} 
 Let $\Omega$ be a bounded domain of $\C^n$ and $A\geq 0, T>0$. Let
  $u,v\in C^{\infty}(\Omega\times (0,T))\cap C(\bar{\Omega}\times [0,T))$ 
 satisfy:
 \begin{itemize}
 \item $u(.,t)$ and $v(.,t)$ are strictly plurisubharmonic functions for any $t\in [0,T)$,
 \item $\dot{u}\leq \log\det (u_{\alpha\bar{\beta}})-Au +f(z,t),$
 \item $\dot{v}\geq \log\det (v_{\alpha\bar{\beta}})-Av+f(z,t),$
 \end{itemize}
 where $f\in C^{\infty}(\bar{\Omega}\times [0,T))$.\\
 Then $\sup\limits_{\Omega\times (0,T)}(u-v)\leq max\{ 0, 
 \sup\limits_{\partial_P(\Omega\times (0,T))}(u-v)\}$.
 \end{Cor}
 \begin{Cor}\label{compa lap}
  Let $\Omega$ be a bounded domain of $\C^n$ and $T>0$. We denote
 by  $L$ the operator on $C^{\infty}(\Omega\times (0,T))$ given by
 $$
 L(f)=\dfrac{\partial f}{\partial t}-\sum a_{\alpha\bar{\beta}}
 \dfrac{\partial^2f}{\partial z_{\alpha}\partial \bar{z}_{\beta}}-b.f,
 $$
 where $a_{\alpha\bar{\beta}}, b\in C(\Omega\times (0,T))$, $(a_{\alpha\bar{\beta}}(z,t))$
 are positive definite Hermitian matrices and $b(z,t)<0$.\\
 Assume that  $\phi\in C^{\infty}(\Omega\times (0,T))\cap C(\bar{\Omega}\times [0,T))$ 
 satisfies\\
 $$L(\phi)\leq  0.$$
 Then $\phi\leq max(0,\sup\limits_{\partial_P(\Omega\times (0,T))}\phi).$
 \end{Cor}
 \subsection{The Laplacian inequalities}
 \begin{flushleft}
 We shall need two standard auxiliary results (see \cite{Yau78}, \cite{Siu87} for a proof).
 \end{flushleft}
 
 \begin{The}
 \label{lap 1}
 Let $\omega_1,\omega_2$ be positive $(1,1)$-forms on a complex manifold $X$.Then\\
 $$
 n\left(\dfrac{\omega_1^n}{\omega_2^n}\right)^{1/n}\leq tr_{\omega_2}(\omega_1)
 \leq n\left(\dfrac{\omega_1^n}{\omega_2^n}\right)(tr_{\omega_1}(\omega_2))^{n-1},
 $$
 where $tr_{\omega_1}(\omega_2)=\dfrac{n\omega_1^{n-1}\wedge\omega_2}{\omega_1^n}$.
 \end{The}
 \begin{Rem}
  Applying Theorem \ref{lap 1} for  $\omega_1=dd^cu$ and $\omega_2=dd^c|z|^2$, we have\\
  $$
  n(\det (u_{\alpha\bar{\beta}}))^{1/n}\leq \Delta u\leq
  n(\det (u_{\alpha\bar{\beta}}))(\sum u^{\alpha\bar{\alpha}})^{n-1} .
  $$
  \end{Rem}
 \begin{The}\label{lap 2}
 Let $\omega, \; \omega '$ be two K\"ahler forms on a complex manifold $X$. If the holomorphic bisectional curvature of $\omega$ is bounded below by a constant $B\in\R$ on $X$,then\\
 $$
 \Delta_{\omega '}\log tr_{\omega}(\omega ')\geq -\frac{tr_{\omega}Ric(\omega ')}{tr_{\omega}(\omega ')}+B\, tr_{\omega '} (\omega),
 $$
 where $Ric(\omega')$ is the form associated  to the Ricci curvature of $\omega'$.
 \end{The}
 \begin{Rem}
 Applying Theorem \ref{lap 2} for $\omega=dd^c|z|^2$ and  $\omega'=dd^cu$, we have\\
 $$
 \sum u^{\alpha\bar{\beta}}(\log \Delta u)_{\alpha\bar{\beta}}\geq\dfrac{\Delta\log\det (u_{\alpha\bar{\beta}})}{\Delta u}.
 $$
 \end{Rem}
 \section{Some properties of weak solutions}
 In this section, we assume that $\Omega$ is a  bounded smooth strictly pseudoconvex domain of $\C^n$,
 $A\geq 0$, $T>0$. 
   
   We will study some properties of the weak solutions of \eqref{KRF}. The proof of Theorem \ref{main.existence} and the proof of Theorem \ref{main.nonsmooth} are contained in this section.
   Theorem \ref{main.existence} is the union of Proposition \ref{exist unique.prop.sec weak} and 
   Corollary \ref{unique.cor.sec weak}. Theorem \ref{main.nonsmooth} is a corollary of Proposition
   \ref{L1weak.prop.secweak} and Proposition \ref{singular.prop.weak sec}.
 \begin{Lem}\label{weaker condition.lem.sec weak}
 Assume that there exists $u_m\in C^{\infty}(\bar{\Omega}\times [0,T))$  satisfying
 \begin{equation}\label{KRF_weak.sec weak}
    \begin{cases}
    \begin{array}{ll}
    u_m(.,t)\in SPSH(\Omega)&\forall t\in [0,T),\\
     u_m(z,t)+2^{-m}\geq u_{m+1}(z,t)&\forall (z,t)\in \bar{\Omega}\times [0,T),\\
    \dot{u}_m=\log\det (u_m)_{\alpha\bar{\beta}}-Au_m+f(z,t)\;\;&\forall (z,t)\in\Omega\times (0,T),\\
    u_m(z,t)\longrightarrow\varphi(z,t)&\forall (z,t)\in\partial\Omega\times [0,T),\\
    u_m(z,0)\longrightarrow u_0(z)&\forall z\in\bar{\Omega}.
    \end{array}
    \end{cases}
    \end{equation}
    Then $u=\lim u_m$ is a weak solution of \eqref{KRF}.
 \end{Lem}
 \begin{proof}
 Set $v_m=u_m+2^{-m+1}e^{-At+AT}$. We have
 $$
 v_m-v_{m+1}=(u_m+2^{-m}-u_{m+1})+2^{-m}(e^{-At+AT}-1)\geq 0.
 $$
 Thus the sequence $\{v_m\}$ is decreasing, and
 \begin{center}
 $v_m\searrow\varphi$ on $\partial\Omega\times [0,T)$,\\
 $v_m\searrow u_0$ on $\bar{\Omega}\times \{0\}$.
 \end{center}
 Moreover, it follows from \eqref{KRF_weak.sec weak} that
 \begin{flushleft}
 $\begin{array}{ll}
 \dot{v}_m&=\dot{u}_m-A.2^{-m+1}e^{-At+AT}\\
 &=\log\det (u_m)_{\alpha\bar{\beta}}-Au_m+f(z,t)-A2^{-m+1}.e^{-At+AT}\\
 &=\log\det (v_m)_{\alpha\bar{\beta}}-Av_m+f(z,t).
 \end{array}$
 \end{flushleft}
 Hence, $u=\lim v_m=\lim u_m$ is a weak solution of \eqref{KRF}.
 \end{proof}
 \begin{Prop}\label{exist unique.prop.sec weak}
 Under the hypotheses of Theorem \ref{main.existence}, there exists a weak solution of \eqref{KRF}.
 \end{Prop}
 \begin{proof}
 Using the convolution of $u_0+\frac{|z|^2}{m}$ with smooth kernels, 
  we can take $u_{0,m}\in C^{\infty}
  (\bar{\Omega})\cap SPSH(\bar{\Omega})$ such that
  \begin{equation}\label{u0m.eq.proof exist.sec weak}
  u_{0,m}\searrow u_0.
    \end{equation}
  Note that $u_0|_{\partial\Omega}$ is continuous. Then 
  \begin{equation}\label{deltam.eq.proof exist.sec weak}
  \delta_m=\sup\limits_{z\in\partial\Omega}(u_{0,m}(z)-u_0(z))
  \stackrel{m\rightarrow\infty}{\longrightarrow} 0.
  \end{equation}
  
  We define $g_m\in C^{\infty}(\bar{\Omega})$ and
   $\varphi_m\in C^{\infty}(\bar{\Omega}\times [0,T))$ by\\
   $$g_m=\log\det (u_{0,m})_{\alpha\bar{\beta}}-Au_{0,m}+f(z,0),$$
   $$\varphi_m=\zeta(\frac{t}{\epsilon_m}) (tg_m+u_{0,m})+(1-\zeta(\frac{t}
   {\epsilon_m}))\varphi,$$
   where $\zeta$ is a smooth function on $\R$ such that $\zeta$ is decreasing,
    $\zeta|_{(-\infty,1]}=1$ and $\zeta|_{[2,\infty)}=0$. $\epsilon_m>0$ are 
  chosen such that  the sequences $\{\epsilon_m\}$,
  $\{\epsilon_m\sup|g_m|\}$ are decreasing to $0$. 
  
  $u_{0,m}$ and $\varphi_m$ satisfy the compatibility condition.
  By Theorem \ref{houli}, there exists $u_m\in C^{\infty}(\bar{\Omega}\times [0,T))$ satisfying 
   \begin{equation}\label{KRF_m.sec weak}
  \begin{cases}
  \begin{array}{ll}
  \dot{u}_m=\log\det (u_m)_{\alpha\bar{\beta}}-Au_m+f(z,t)\;\;\;&\mbox{on}\;\Omega\times (0,T),\\
  u_m=\varphi_m&\mbox{on}\;\partial\Omega\times [0,T),\\
  u_m=u_{0,m}&\mbox{on}\;\bar{\Omega}\times\{ 0\}.\\
  \end{array}
  \end{cases}
  \end{equation}
    It follows from \eqref{deltam.eq.proof exist.sec weak} that, for any $m>0$,

\begin{flushleft}
$\varphi_m\geq\zeta(\frac{t}{\epsilon_m}) u_0+(1-\zeta(\frac{t}
       {\epsilon_m}))\varphi-2\epsilon_m\sup|g_m|,$\\
       $\varphi_m\leq \zeta(\frac{t}{\epsilon_m}) u_0+(1-\zeta(\frac{t}
                   {\epsilon_m}))\varphi+2\epsilon_m\sup|g_m|+\delta_m.$
\end{flushleft}
 where $(z,t)\in\partial\Omega\times [0,T)$.
 
 Then  
 $$\sup\limits_{\partial\Omega\times [0,T)}|\varphi_m-\varphi|\leq 
 \sup\limits_{\partial\Omega\times [0,2\epsilon_m]}|\varphi-u_0|+2\epsilon_m\sup|g_m|+\delta_m.$$
 Note that $u_0(z)=\varphi(z,0)$ for any $z\in\partial\Omega$. Hence
 $$\sup\limits_{\partial\Omega\times [0,T)}|\varphi_m-\varphi|
 \stackrel{m\rightarrow \infty}{\longrightarrow} 0.$$
 Then we can choose a subsequence $\{\varphi_{m_k}\}$ such that
 \begin{equation}\label{varphimstab.eq.proof exist.sec weak}
 \sup\limits_{\partial\Omega\times [0,T)}|\varphi_{m_k}-\varphi|\leq 2^{-k-1}.
 \end{equation}
 for any $k>0$.
 
 Using \eqref{u0m.eq.proof exist.sec weak}, \eqref{varphimstab.eq.proof exist.sec weak}
 and applying Corollary \ref{compa log}, we have
 $$u_{m_k}+2^{-k}\geq u_{m_{k+1}}.$$ 
 It follows from Lemma \ref{weaker condition.lem.sec weak} that $u=\lim u_{m_k}$ is a weak solution of
 \eqref{KRF}.
 \end{proof}
 \begin{Prop}\label{compa.prop.weak sec}
 Assume that $u$ is a weak solution of \eqref{KRF} and $v$ is a weak solution of
 \begin{equation}\label{compa.eq.sec weak}
   \begin{cases}
   \begin{array}{ll}
   \dot{v}=\log\det (v_{\alpha\bar{\beta}})-Av+g(z,t)\;\;\;&\mbox{on}\;\Omega\times (0,T),\\
   v=\psi&\mbox{on}\;\partial\Omega\times [0,T),\\
   v=v_0&\mbox{on}\;\bar{\Omega}\times\{ 0\},\\
   \end{array}
   \end{cases}
   \end{equation}
   where $v_0$ is a plurisubharmonic function in a neighbourhood of
    $\bar{\Omega}$ and $g,\psi$ are smooth functions 
   in $\bar{\Omega}\times [0,T]$. If there are $A_1,A_2,A_3\geq 0$ such that 
   $$u_0\leq v_0+A_1,$$
   $$\varphi|_{\partial\Omega\times (0,T)}\leq \psi|_{\partial\Omega\times (0,T)}+A_2,$$
   $$f\leq g+A_3,$$
   then $u\leq v+\max \{A_1,A_2\}+A_3T$.
 \end{Prop}
 \begin{proof}
 Assume that $u_m\in C^{\infty}(\bar{\Omega}\times [0,T))$ satisfies 
 \begin{equation}\label{um.eq.sec weak}
     \begin{cases}
     \begin{array}{ll}
     u_m(.,t)\in SPSH(\Omega)\\
     \dot{u}_m=\log\det (u_m)_{\alpha\bar{\beta}}-Au_m+f(z,t)\;\;\;&\mbox{on}\;\Omega\times (0,T),\\
     u_m\searrow\varphi&\mbox{on}\;\partial\Omega\times [0,T),\\
     u_m\searrow u_0&\mbox{on}\;\bar{\Omega}\times\{ 0\},
     \end{array}
     \end{cases}
     \end{equation}
  and that $v_m\in C^{\infty}(\bar{\Omega}\times [0,T))$ satisfies
 \begin{equation}
    \begin{cases}
    \begin{array}{ll}
    v_m(.,t)\in SPSH(\Omega)\\
    \dot{v}_m=\log\det (v_m)_{\alpha\bar{\beta}}-Av_m+g(z,t)\;\;\;&\mbox{on}\;\Omega\times (0,T),\\
    v_m\searrow\psi&\mbox{on}\;\partial\Omega\times [0,T),\\
    v_m\searrow v_0&\mbox{on}\;\bar{\Omega}\times\{ 0\}.
    \end{array}
    \end{cases}
    \end{equation}
    Fix $m>0,\epsilon>0,0<T'<T$. We need to show that there exists $k_m>0$ satisfying
    \begin{equation}\label{weak compa.eq.sec weak}
    u_{k_m}\leq v_m+\max\{A_1,A_2\}+A_3T+\epsilon,\;\;\forall  (z,t)\in \Omega\times (0,T').
    \end{equation}
    Indeed, if we denote $w_m=v_m+A_3t$ then
    $$\dot{w}_m\geq \log\det (w_m)_{\alpha\bar{\beta}}-Aw_m+g(z,t)+A_3\geq 
    \log\det (w_m)_{\alpha\bar{\beta}}-Aw_m+f(z,t)$$
    It follows from Corollary \ref{compa log} that
    $$u_k-v_m=u_k-w_m+A_3t\leq \sup\limits_{\partial_P(\Omega\times (0,T'))} (u_k-w_m)+A_3T,\;\;\forall k>0,
    (z,t)\in \Omega\times (0,T').$$
    Note that
    $$\cap_{k>0}\{(z,t)\in\partial\Omega\times [0,T']:u_k(z,t)\geq w_m(z,t)+A_2+\epsilon\}=\emptyset.$$ 
    By the compactness of $\partial\Omega\times [0,T']$, there exists $k^{'}_m>0$ such that
    $$
    \cap_{k< k^{'}_m}\{(z,t)\in\partial\Omega\times [0,T']:u_k(z,t)\geq w_m(z,t)+A_2+\epsilon\}=\emptyset.
    $$
    By the monotonicity of $\{u_k\}$, we have
    $$u_{k^{'}_m}(z,t)<w_m(z,t)+A_2+\epsilon,\;\;\forall (z,t)\in\partial \Omega\times [0,T'].$$
    Similarly, there exists $k^{''}_m>0$ such that
    $$u_{k^{''}_m}(z,0)<w_m(z,0)+A_1+\epsilon,\;\;\forall z\in\bar{\Omega}.$$
    Denote $k_m=\max\{k^{'}_m,k^{''}_m\}$. We have
    $$\sup\limits_{\partial_P(\Omega\times (0,T'))} (u_{k_m}-w_m)\leq \max\{A_1,A_2\}+\epsilon.$$
    Then 
    $$u_{k_m}\leq v_m+\max\{A_1,A_2\}+A_3T+\epsilon,\;\;\forall  (z,t)\in \Omega\times (0,T').$$
    When $m\rightarrow \infty$, we obtain
    $$u\leq v+\max\{A_1,A_2\}+A_3T+\epsilon,\;\;\forall  (z,t)\in \Omega\times (0,T').$$
    When $\epsilon\rightarrow 0$ and $T'\rightarrow T$, we obtain
    $$u\leq v+\max\{A_1,A_2\}+A_3T, \;\;\forall  (z,t)\in \Omega\times (0,T).$$
 \end{proof}
 \begin{Cor}\label{unique.cor.sec weak}
 The weak solution of \eqref{KRF} is unique.
 \end{Cor}
 \begin{Rem}
 By Proposition \ref{exist unique.prop.sec weak} and Corollary \ref{unique.cor.sec weak}, if 
 $0<T'<T$ and $v$ is the 
 weak solution of \eqref{KRF} on $\bar{\Omega}\times [0,T')$ then
  $$v=u|_{\bar{\Omega}\times [0,T')},$$
  where $u$ is the  weak solution of \eqref{KRF} on $\bar{\Omega}\times [0,T)$.
 \end{Rem}
 \begin{Prop}\label{stab.prop.weak sec}
 Assume that $\psi_m,g_m$ are smooth functions in $\bar{\Omega}\times [0,T]$ and $v_{0,m}$ is a plurisubharmonic
 function in a neighbourhood of $\bar{\Omega}$ satisfying
\begin{center}
 $\begin{cases}
 \psi_m\searrow \varphi,\; g_m\searrow f,\; v_{0,m}\searrow u_0, \\
  v_{0,m}|_{\partial\Omega}=\psi_m(.,0)|_{\partial\Omega}.
 \end{cases}$
 \end{center}
 Suppose that $v_m\in USC(\bar{\Omega}\times [0,T))$ is the weak solution of 
 \begin{center}
 $\begin{cases}
 \begin{array}{ll}
  \dot{v}_m=\log\det (v_m)_{\alpha\bar{\beta}}-Av_m+g_m(z,t)\;\;\;&\mbox{on}\;\Omega\times (0,T),\\
  v_m=\psi_m&\mbox{on}\; \partial\Omega\times [0,T),\\
  v_m=v_{0,m}&\mbox{on}\; \bar{\Omega}\times \{0\}.
  \end{array}
 \end{cases}$
 \end{center}
 Then $v_m\searrow u$, where $u$ is the weak solution of \eqref{KRF}.
 \end{Prop}
 \begin{proof}
 By Proposition \ref{compa.prop.weak sec}, we have
 $$v_1\geq v_2\geq...\geq v_m\geq ...\geq u.$$
 We need to show that $\lim v_m\leq u$.\\
 Let $0<T'<T$ and $\epsilon>0$. By Dini's theorem, there exists $m_1>0$ such that 
 \begin{center}
 $\psi_{m_1}<\varphi+\epsilon$ on $\partial\Omega\times [0,T']$,\\
 $g_{m_1}<f+\epsilon$ on $\bar{\Omega}\times [0,T']$.
 \end{center} 
 Assume that $u_m\in C^{\infty}(\bar{\Omega}\times [0,T))$ satisfies \eqref{um.eq.sec weak}. Fix $k>0$.
 Let $h$ be a harmonic function in $\bar{\Omega}$ such that $h|_{\partial\Omega}=u_k(.,0)|_{\partial\Omega}$.
 Then there exists a subset $K\Subset\Omega$ such that
             $$(h-u_k(.,0))|_{\bar{\Omega}\setminus K}\leq \epsilon.$$
  Note that 
  $$v_{0,m_1}-h\leq \sup\limits_{\partial\Omega}(v_{0,m_1}-h)
  \leq \sup\limits_{\partial\Omega}(v_{0,m_1}-u_0)\leq \epsilon.$$
  Then $(v_{0,m_1}-u_k(.,0))|_{\bar{\Omega}\setminus K}\leq 2\epsilon$.
  
  Moreover, by Hartogs lemma, there exists $m_2>0$ such that
  $$(v_{0,m_2}-u_k(.,0))|_K\leq \epsilon.$$
  Then, for any $m>\max\{m_1,m_2\}$, we have
  \begin{center}
   $\psi_{m}<u_k+\epsilon$ on $\partial\Omega\times [0,T']$,\\
   $g_{m}<f+\epsilon$ on $\bar{\Omega}\times [0,T']$,\\
   $v_{0,m}\leq u_k (.,0)+2\epsilon$ on $\bar{\Omega}$.
   \end{center} 
  It follows from Proposition \ref{compa.prop.weak sec} that
  \begin{center}
  $v_m(z,t)\leq u_k(z,t)+(T+2)\epsilon$ on $\Omega\times (0,T')$.
  \end{center}
  Then 
  \begin{center}
  $\lim\limits_{m\rightarrow \infty} v_m(z,t)\leq u_k(z,t)+(T+2)\epsilon$ on $\Omega\times (0,T')$.
  \end{center}
  When $k\rightarrow\infty$, $\epsilon\rightarrow 0$ and $T' \rightarrow T$, we obtain
  $$\lim\limits_{m\rightarrow \infty }v_m \leq u.$$
 \end{proof}
 \begin{Lem}\label{utminusu0.lem.weak sec}
 Suppose that $\psi,g\in C^{\infty}(\bar{\Omega}\times [0,T])$. 
 Assume that $v\in C^{\infty}(\bar{\Omega}\times [0,T))$ satisfies
 \begin{equation}
    \begin{cases}
    \begin{array}{ll}
    v(.,t)\in SPSH(\bar{\Omega}),\\
    \dot{v}=\log\det (v_{\alpha\bar{\beta}})-Av+g(z,t)\;\;\;&\mbox{on}\;\Omega\times (0,T),\\
    v=\psi&\mbox{on}\;\partial\Omega\times [0,T).\\
    \end{array}
    \end{cases}
    \end{equation}
    Then $$v(z,t)-v(z,0)\geq -C(t),$$ for any $(z,t)\in \bar{\Omega}\times [0,T)$. Here $C(t)$ is defined by
    $$C(t)=\inf\limits_{1>\epsilon>0}((-n\log\epsilon+A\sup |\psi|+
    \sup |g|)t-\epsilon\inf\rho)+
    \sup\limits_{t'\in [0,t]}\sup\limits_{z\in\partial\Omega}|\psi (z,t')-\psi (z,0)|,$$
    where $\rho\in C^{\infty}(\bar{\Omega})$ such that $dd^c\rho\geq dd^c|z|^2$ and $\rho|_{\partial\Omega}=0$.
     \end{Lem}
     \begin{proof}
     Fix $0<\epsilon<1$ and $0<t_0<T$. Denote
     $$A_{\epsilon}=-n\log\epsilon+A\sup |\psi|+ \sup |g|,$$
    $$\delta_{t_0}=\sup\limits_{t\in [0,t_0]}\sup\limits_{z\in\partial\Omega}|\psi (z,t)-\psi (z,0)|.$$
     We consider 
     $$w(z,t)=v(z,0)-A_{\epsilon}t+\epsilon\rho (z)-\delta_{t_0}.$$
     We have
     \begin{flushleft}
     $\begin{array}{ll}
     \dot{w}-\log\det w_{\alpha\bar{\beta}}+Aw-g
     &\leq -A_{\epsilon}-n\log\epsilon-\log\det\rho_{\alpha\bar{\beta}}+A\sup w+\sup |g|\\
     &\leq -A_{\epsilon}-n\log\epsilon+A\sup \psi+\sup |g|\\
     &\leq 0.
     \end{array}$
     \end{flushleft}
     Moreover,
     \begin{center}
     $w(z,0)\leq v(z,0)$ on $\bar{\Omega}$,\\
     $w(z,t)\leq v(z,0)-\delta_{t_0}\leq v(z,t)$ on $\partial\Omega\times [0,t_0)$.
     \end{center}
     Applying Corollary \ref{compa log}, we obtain
     $$v(z,t)\geq w(z,t)=v(z,0)-A_{\epsilon}t+\epsilon \rho(z)-\delta_{t_0},\;\;
     \forall (z,t)\in\bar{\Omega}\times [0,t_0).$$
     Thus
     $$v(z,t_0)\geq v(z,0)-\inf\limits_{1>\epsilon>0}\{A_{\epsilon}t_0-\epsilon \inf\rho\}-\delta_{t_0}.$$
     \end{proof}
  \begin{Rem} Under the conditions of Lemma \ref{utminusu0.lem.weak sec}, 
  $C(t)$ is an increasing function satisfying
  $$\lim\limits_{t\rightarrow 0}C(t)=0.$$
  \end{Rem}
 \begin{Prop}\label{L1weak.prop.secweak}
 If a function $u\in C^{\infty}(\bar{\Omega}\times (0,T))$ satisfies\\
 \begin{equation}\label{L1weak.eq. sec weak}
  \begin{cases}
  \begin{array}{ll}
  u(.,t)\in SPSH(\Omega),\\
  \dot{u}=\log\det (u_{\alpha\bar{\beta}})-Au+f(z,t)\;\;\;&\mbox{on}\;\Omega\times (0,T),\\
  u=\varphi&\mbox{on}\;\partial\Omega\times [0,T),\\
  u(.,t)\stackrel{L^1}{\longrightarrow}u_0.
  \end{array}
  \end{cases}
  \end{equation}
  Then $u$ is the weak solution of \eqref{KRF}.
 \end{Prop}
 \begin{proof}
 By Lemma \ref{utminusu0.lem.weak sec}, there exists $t_m\searrow 0$ such that
 \begin{equation}\label{initial.eq.proof L1}
 u(z,t_m+t)\geq u(z,t_{m+1}+t)-2^{-m},\;\;\forall (z,t)\in \bar{\Omega}\times [0,T-t_m).
 \end{equation}
 By the condition "$u(.,t)\stackrel{L^1}{\rightarrow}u_0$", we have 
 $$u(.,t_m)+2^{-m+1}\searrow u_0.$$
 Passing to a subsequence, we can assume that
 $$f(z,t_m+t)+2^{-m+1}\searrow f(z,t),$$
 $$\varphi (z,t_m+t)+2^{-m+1}\searrow \varphi (z,t).$$
 Fix $k>0$. 
 For any $m>k$, $u_m=u(.,t_m+.)\in C^{\infty}(\bar{\Omega})\times [0,T-t_k))$ is the solution of equation
 \begin{equation}
 \begin{cases}
 \begin{array}{ll}
 \dot{u}_m=\log\det(u_m)_{\alpha\bar{\beta}}-Au_m+f(z,t_m+t)\;\;&\forall (z,t)\in\Omega\times (0,T-t_k),\\
 u_m(z,0)=u(z,t_m),&\forall z\in\bar{\Omega},\\
 u_m(z,t)=u(z,t_m+t),&\forall (z,t)\in\partial\Omega\times [0,T-t_k).
 \end{array}
 \end{cases}
 \end{equation}
 Let $v_m\in USC(\bar{\Omega}\times [0,T-t_k))$ be the weak solution of equation
 \begin{equation}
  \begin{cases}
  \begin{array}{ll}
  \dot{v}_m=\log\det(v_m)_{\alpha\bar{\beta}}-Av_m+f(z,t_m+t)+2^{-m+1}
  &\mbox{ on }\Omega\times (0,T-t_k),\\
  v_m(z,0)=u(z,t_m)+2^{-m+1},&\mbox{ on } \bar{\Omega},\\
  v_m(z,t)=\varphi (z,t_m+t)+2^{-m+1},&\mbox{ on }\partial\Omega\times [0,T-t_k).
  \end{array}
  \end{cases}
  \end{equation}
  Applying Proposition \ref{stab.prop.weak sec}, we have 
  $$v_m\searrow v,$$
  where $v$ is the weak solution of  \eqref{KRF} on $\bar{\Omega}\times [0,T-t_k)$.
  
  Applying Proposition \ref{compa.prop.weak sec}, we have 
  $$\sup\limits_{\bar{\Omega}\times [0,T-t_k)} |u_m-v_m|\stackrel{m\rightarrow 0}{\longrightarrow} 0.$$
  Then
  $$u(z,t)=\lim u_m(z,t)=\lim v_m(z,t)=v(z,t),\;\;\forall (z,t)\in\bar{\Omega}\times [0,T-t_k).$$
  Hence $u$ is the weak solution of \eqref{KRF} on $\bar{\Omega}\times [0,T-t_k)$.
  
  When $k\rightarrow\infty$, 
  we have $u$ is the weak solution of \eqref{KRF} on $\bar{\Omega}\times [0,T)$.
 \end{proof}
 \begin{Prop}\label{singular.prop.weak sec}
 If there is $a\in\Omega$ such that $\nu_{u_0}(a)>0$, then the weak solution $u$ of \eqref{KRF} satisfies $u(a,t)=-\infty$ for $t\in [0,\epsilon_A(\nu_{u_0}(a)))$.
 \end{Prop}
 \begin{proof}
  We need to show that $u(a,t)=-\infty$ for $t\in [0,\epsilon_A(\nu))$ when $0<\nu<\nu_{u_0}(a)$,
  so that we have
   \begin{equation}\label{1.proof nonexist. sec weak}
 \begin{cases}
 u_0\leq \nu\log|z-a|+B_1,\;\;\;\forall z\in \bar{\Omega},\\
 \varphi\leq \nu\log|z-a|+B_1,\;\;\;\forall (z,t)\in\partial\Omega\times [0,T),
 \end{cases}
 \end{equation}
 where $B_1>0$ is given. 
 
 Let $\chi: \R\rightarrow \R_{+}$ be a smooth increasing convex function such that $\chi|_{(-\infty,-1)}=0$,
 $\chi|_{(1,\infty)}=Id$. For any $m$, we denote
 $$w_m(z)=\chi(\log|z-a|+m)-m.$$
 We will show that there exists $B>0$ such that
 $$v_m(z,t)=g(t)w_m(z)+|z|^2+B(t+1)\geq u(z,t),$$
 for any $m>0$ and $(z,t)\in\bar{\Omega}\times [0,\epsilon_A(\nu))$. Here $g(t)=k_A(\nu,t)$ 
 as in \eqref{kA.eq}.

 It is easy to show that
 \begin{center}
 $\dot{v}_m(z,t)+Av_m(z,t)\geq -2n w_m(z)+B,\;\;\;\forall (z,t)\in \bar{\Omega}\times [0,\epsilon_A(\nu))$.
 \end{center}
 When $|z-a|< e^{-m-1}$, we have
 $w_m=-m, |D^2w_m|=0.$ Then 
 \begin{equation}\label{2.proof nonexist. sec weak}
 \dot{v}_m-\log\det (v_m)_{\alpha\bar{\beta}}+Av_m-f(z,t)\geq 2mn+B-f(z,t)
 \end{equation}
 When $|z-a|\geq e^{-m-1}$, we have
 \begin{flushleft}
 $w_m\leq \log |z-a|+2,$\\
 $|(w_m)_{\alpha\bar{\beta}}|\leq \frac{B_2}{|z-a|^2},$
 \end{flushleft}
 where $B_2>0$ is independent of $m$. Then
 $$\log\det (v_m)_{\alpha\bar{\beta}}=\log\det(g(t)w_m+|z|^2)_{\alpha\bar{\beta}}\leq \log\frac{B_3}{|z-a|^{2n}},$$
 where $B_3>0$ is independent of $m$.
 
 Hence
 \begin{equation}\label{3.proof nonexist. sec weak}
 \dot{v}_m-\log\det (v_m)_{\alpha\bar{\beta}}+Av_m-f(z,t)\geq -4n-\log B_3+B-f(z,t).
 \end{equation}
 By \eqref{1.proof nonexist. sec weak}, \eqref{2.proof nonexist. sec weak} and
 \eqref{3.proof nonexist. sec weak}, there exists $B>0$ such that, for any $m$,
 \begin{equation}
 \begin{cases}
 \begin{array}{ll}
\dot{v}_m-\log\det (v_m)_{\alpha\bar{\beta}}+Av_m-f(z,t)\geq 0 &\forall (z,t)\in
\Omega\times (0,\epsilon_A(\nu)),\\
v_m(z,0)\geq u_0(z),&\forall z\in\bar{\Omega},\\
v_m(z,t)\geq \varphi (z,t),&\forall (z,t)\in\partial\Omega\times [0,\epsilon_A(\nu)).
 \end{array}
 \end{cases}
 \end{equation}
 Applying Proposition \ref{compa.prop.weak sec}, we have
 \begin{center}
 $v_m\geq u$ on $\bar{\Omega}\times [0,\epsilon_A(\nu))$.
 \end{center}
 When $m\rightarrow\infty$, we obtain
 \begin{center}
 $g(t)\log|z-a|+|z|^2+B(t+1)\geq u(z,t),\;\;\forall (z,t)\in\bar{\Omega}\times [0,\epsilon_A(\nu)).$
 \end{center}
 In particular, $u(a,t)=-\infty$ for $t\in [0,\epsilon_A(\nu))$.
 \end{proof}
 \section{Proof of Theorem \ref{main.zeroLelong}}
 \subsection{Bounds on $\mathbf{\dot{u}}$}
  \begin{Lem}\label{dotu.lem}
  Assume that $u\in C^{\infty}(\bar{\Omega}\times [0,T))$ satisfies
  \begin{equation}
   \begin{cases}
   \dot{u}=\log\det (u_{\alpha\bar{\beta}})-Au+g(z,t)\;\;\mbox{on}\;\; \Omega\times (0,T),\\
   u=\varphi\;\;\;\mbox{on}\;\; \partial\Omega\times [0,T).
   \end{cases}
   \end{equation}
   Then the following hold\\
   
  (i) If $A=0$ then 
  $$\dfrac{u(z,t)-\sup u_0}{t}-B\leq \dot{u}(z,t)\leq \dfrac{u(z,t)-u_0(z)}{t}+B, \;\;\forall 
  (z,t)\in\bar{\Omega}\times [0,T),$$
  where $B= 2\sup |\dot{\varphi}|+T\sup |\dot{g}|+n$ and $u_0=u(., 0)$.\\
  
  (ii) If $A>0$ then 
  $$\dfrac{A}{e^{At}-1}(u-e^{At}\sup (u_0)_{+})-B
   \leq \dot{u}\leq \dfrac{A}{e^{At}-1} (u-u_0)+B,$$
   where $B= 2\sup |\dot{\varphi}|+T\sup |\dot{g}|+n$ and $u_0=u(., 0)$ and $(u_0)_{+}=\max\{u_0, 0\}$.
  \end{Lem}
  \begin{proof}
   We denote by $L$ the operator
  $$L(\phi)=\dot{\phi}-\sum u^{\alpha\bar{\beta}}\phi_{\alpha\bar{\beta}}+A\phi,\;\;\forall
  \phi\in C^{\infty}(\bar{\Omega}\times [0,T)),$$
  where $(u^{\alpha\bar{\beta}})$ is the transpose of the
  inverse of the Hessian matrix  $(u_{\alpha\bar{\beta}})$.\\
  
  (i) When $A=0$, we have, for any $B>0$,
  \begin{flushleft}
  $\begin{array}{ll}
  L(t\dot{u}-u+u_0-Bt)&=t\ddot{u}+\dot{u}-\dot{u}-B-
  t\sum u^{\alpha\bar{\beta}}\dot{u}_{\alpha\bar{\beta}}
  +\sum u^{\alpha\bar{\beta}}u_{\alpha\bar{\beta}}
  -\sum u^{\alpha\bar{\beta}}(u_0)_{\alpha\bar{\beta}}\\
  &=tg_t-B+n-\sum u^{\alpha\bar{\beta}}(u_0)_{\alpha\bar{\beta}}\\
  &\leq T\sup |\dot{g}|-B+n.\\
  L(t\dot{u}-u+Bt)&=t\ddot{u}+\dot{u}-\dot{u}+B-t\sum u^{\alpha\bar{\beta}}\dot{u}_{\alpha\bar{\beta}}
  +\sum u^{\alpha\bar{\beta}}u_{\alpha\bar{\beta}}\\
  &=B+n+t\dot{g}\\
  &\leq B+n-T\sup |\dot{g}|.
  \end{array}$
  \end{flushleft}
  Then, for any $B\geq T\sup |\dot{g}|+n$, we have 
  \begin{center}
  $L(t\dot{u}-u+u_0-Bt)\leq 0,$\\
  $L(t\dot{u}-u+Bt)\geq 0.$
  \end{center}
   It follows from Corollary \ref{compa lap} that, for any $B\geq 2\sup |\dot{\varphi}|
   +T\sup |\dot{g}|+n$,
   \begin{center}
  $t\dot{u}-u+u_0-Bt\leq \sup_{\partial_P(\Omega\times [0,T))}(t\dot{u}-u+u_0-Bt)\leq 0.$
  $t\dot{u}-u+Bt\geq \inf_{\partial_P(\Omega\times [0,T))}(t\dot{u}-u+Bt)\geq -\sup u_0.$
  \end{center}
  Hence, if $B= 2\sup |\dot{\varphi}|+T\sup |\dot{g}|+n$ then
  $$\dfrac{u(z,t)-\sup u_0}{t}-B\leq \dot{u}(z,t)\leq \dfrac{u(z,t)-u_0(z)}{t}+B, 
  $$
  for any $(z,t)\in\bar{\Omega}\times [0,T)$.\\
  
  (ii)When $A>0$, we have
  \begin{flushleft}
  $\begin{array}{ll}
  L(\dfrac{1-e^{-At}}{A}\dot{u}-e^{-At}u)
  &=\dfrac{1-e^{-At}}{A}\ddot{u}+e^{-At}\dot{u}-e^{-At}\dot{u}+Ae^{-At}u
  -\dfrac{1-e^{-At}}{A}\sum u^{\alpha\bar{\beta}}\dot{u}_{\alpha\bar{\beta}}\\
  &+e^{-At}\sum u^{\alpha\bar{\beta}}u_{\alpha\bar{\beta}}
  +A(\dfrac{1-e^{-At}}{A}\dot{u}-e^{-At}u)\\
  &=\dfrac{1-e^{-At}}{A}\dot{g}+e^{-At}n,\\
  L(e^{-At}u_0)&=-Ae^{-At}u_0-e^{-At}\sum u^{\alpha\bar{\beta}}(u_0)_{\alpha\bar{\beta}}
  +Ae^{-At}u_0\leq 0,\\
  L(Bt)&=B+ABt,
  \end{array}$
  \end{flushleft}
  where $B>0$.\\
  
  Then, for any $B\geq T\sup |\dot{g}|+n$, we have
  \begin{center}
  $L(\dfrac{1-e^{-At}}{A}\dot{u}-e^{-At}(u-u_0)-Bt)\leq 0,$\\
  $L(\dfrac{1-e^{-At}}{A}\dot{u}-e^{-At}u+Bt)\geq 0.$
  \end{center}
  It follows from Corollary \ref{compa lap} that, for any $B\geq 2\sup |\dot{\varphi}|+
  T\sup |\dot{g}|+n$,
  \begin{flushleft}
  $\dfrac{1-e^{-At}}{A}\dot{u}-e^{-At}(u-u_0)-Bt\leq
  \sup\limits_{\partial_P(\Omega\times [0,T))}(\dfrac{1-e^{-At}}{A}\dot{u}-e^{-At}(u-u_0)-Bt)\leq 0,$\\
  \end{flushleft}
  and
  \begin{flushleft}
  $\begin{array}{ll}
  \dfrac{1-e^{-At}}{A}\dot{u}-e^{-At}u+Bt\geq
  \inf\limits_{\partial_P(\Omega\times [0,T))}(\dfrac{1-e^{-At}}{A}\dot{u}-e^{-At}u+Bt)
  &\geq -\sup(e^{-At}u_0)\\
  &\geq -\sup (u_0)_{+}.
  \end{array}$
  \end{flushleft}
   Hence, if $B= 2\sup |\dot{\varphi}|+T\sup |\dot{g}|+n$ then
   $$\dfrac{A}{e^{At}-1}\left(u-e^{At}\sup (u_0)_{+}\right)-B\dfrac{At}{1-e^{-At}}
   \leq \dot{u}\leq \dfrac{A}{e^{At}-1} (u-u_0)+B\dfrac{At}{1-e^{-At}}.$$
   Thus
   $$\dfrac{A}{e^{At}-1}(u-e^{At}\sup (u_0)_{+})-B
   \leq \dot{u}\leq \dfrac{A}{e^{At}-1} (u-u_0)+B.$$
  \end{proof}
  \subsection{Proof of Theorem \ref{main.zeroLelong}}
  \subsubsection{Smoothness}
  \begin{flushleft}
   As in the proof of Proposition \ref{exist unique.prop.sec weak}, we can construct a sequence of functions
   $u_m\in C^{\infty}(\bar{\Omega}\times [0,T))$ satisfying
   \end{flushleft}
   \begin{equation}
      \begin{cases}
      \begin{array}{ll}
      u_m(.,t)\in SPSH(\Omega),\\
      \dot{u}_m=\log\det (u_m)_{\alpha\bar{\beta}}-Au_m+f(z,t)\;\;\;&\mbox{on}\;\Omega\times (0,T),\\
      u_m=\varphi_m\rightrightarrows\varphi&\mbox{on}\;\partial\Omega\times [0,T),\\
      u_m=\varphi &\mbox{on}\;\partial\Omega\times (\epsilon_m,T),\\
      u_m=u_{0,m}\searrow u_0&\mbox{on}\;\bar{\Omega}\times\{ 0\},\\
      u=\lim\limits_{m\rightarrow \infty} u_m.
      \end{array}
      \end{cases}
      \end{equation}
      where $\epsilon_m\searrow 0$.
      
 $\forall\epsilon\in (0,T), \forall m\gg 1$, $u_m(.,\epsilon)$  verifies
 \begin{equation}
 \begin{cases}
 (dd^cu_m(z,\epsilon))^n=F_m(z)\;\;\mbox{ on }\;\Omega,\\
 u_m(z,\epsilon)=\varphi (z,\epsilon)\;\;\mbox{ on }\;\partial\Omega,
 \end{cases}
 \end{equation}
 where $F_m(z)=\exp (\dot{u}_m(z,\epsilon)+Au_m(z,\epsilon)-f(z,\epsilon))$.
 
 By Lemma \ref{dotu.lem} and Proposition \ref{singular.prop.weak sec}, we have, for $m\gg 1$,
 \begin{center}
 $F_m(z)\leq C_1e^{-\frac{u(z,\epsilon/2)}{\epsilon /2}}\leq C_2e^{-\frac{2u_0(z)}{\epsilon}}$
 \end{center}
 where $C_1, C_2>0$ depend only on $\Omega, \epsilon, T, A, f, \varphi$.
 
 By  Skoda's theorem (see \cite{Sko72}), we have $F_m\in L^p(\bar{\Omega})$ for any $p>1$.
 
 Applying Kolodziej theorem (Theorem B \cite{Kol98}), we have 
 $$\|u_m(.,\epsilon)\|_{L^{\infty}(\bar{\Omega})}\leq C.$$
 where $C$ is independent of $m$.
 
 Hence, the weak solution $u$ satisfies
 $$\|u(.,\epsilon)\|_{L^{\infty(\bar{\Omega})}}\leq C.$$
 By Corollary \ref{unique.cor.sec weak}, Proposition \ref{L1weak.prop.secweak} and
 the case where $u_0$ is bounded (\cite{Do15}), we have $u$ is smooth on $\bar{\Omega}\times (\epsilon,T)$
 and verifies
 \begin{equation}\label{proofzeroLelong.PMA.eq}
 \dot{u}=\log\det (u_{\alpha\bar{\beta}})-Au+f(z,t) 
 \end{equation}
 on $\bar{\Omega}\times (\epsilon,T)$.
 
 When $\epsilon\searrow 0$, we have $u$ is smooth on $\bar{\Omega}\times (0,T)$
  and verifies \eqref{proofzeroLelong.PMA.eq} on  $\bar{\Omega}\times (0,T)$.
  \subsubsection{Continuity at zero}
  \begin{flushleft}
   Applying Lemma \ref{utminusu0.lem.weak sec}, we have
   \end{flushleft}
   \begin{equation}\label{liminf.conti0.proofzero}
   \liminf\limits_{t\to 0}u(z,t)\geq u_0(z),
   \end{equation}
   for any $z\in\bar{\Omega}$.
   
   Note that $u$ is the limit of a decreasing sequence of smooth functions, then $u\in USC(\bar{\Omega}\times [0,T))$. We have
   \begin{equation}\label{limsup.conti0.proofzero}
   \limsup\limits_{t\to 0}u(z,t)\leq u_0(z),
   \end{equation}
   for any $z\in\bar{\Omega}$.
   
   Combining \eqref{liminf.conti0.proofzero} and \eqref{limsup.conti0.proofzero}, we obtain
   $$\lim\limits_{t\to 0}u(z,t)=u_0(z).$$
   
   Hence, by the dominated convergence theorem, $u(., t)\rightarrow u_0$ in $L^1$, as $t\to 0$.
 \section{A priori estimates}
 In this section, we will prove a priori estimates which will 
 be used in the proof of Theorem \ref{main.description}.

 We suppose that $\Omega$ be a bounded smooth strictly pseudoconvex domain of $\C^n$,
  $C_0,A,T, N_1,...,N_l>0$ and
 $a_1,...a_l\in\Omega$. 
 We also suppose that $\varphi, g$ are smooth functions in $\bar{\Omega}\times [0,T]$ and 
 $u_0$ is a plurisubharmonic function in $\Omega$ satisfying
 \begin{equation}\label{condition.g phi.a priori}
 \|g\|_{C^2(\bar{\Omega}\times[0,T])}+\|\varphi\|_{C^4(\bar{\Omega}\times[0,T])}\leq C_0,
 \end{equation}
 \begin{equation}\label{condition.u0.a priori}
 u_0\geq \sum\limits_{j=1}^lN_j\log |z-a_j|-C_0,
 \end{equation}
 where $\|.\|_{C^k(\bar{\Omega}\times[0,T])}$ is defined by
 $$\|\phi\|_{C^k(\bar{\Omega}\times[0,T])}=\sum\limits_{|j_1|+j_2\leq 2}\sup_{\bar{\Omega}\times [0,T]}
 |D_x^{j_1}D_t^{j_2}\phi|,$$
 for any $\phi\in C^{\infty}(\bar{\Omega}\times [0,T])$. 
 
 Throughout this section, unless  specified otherwise, we always assume that $u_0$ is smooth and strictly plurisubharmonic in $\bar{\Omega}$. 
  The main result of this section is following
 \begin{The}\label{the.a priori}
  Assume that  
 $u\in C^{\infty}(\bar{\Omega}\times [0,T))$ satisfies
 \begin{equation}\label{eq.the.a priori}
 \begin{cases}
 \dot{u}=\log\det (u_{\alpha\bar{\beta}})-Au+g(z,t)\;\;\mbox{on}\;\; \Omega\times (0,T),\\
 u=\varphi\;\;\;\mbox{on}\;\; \partial\Omega\times [0,T),\\
 u(.,0)=u_0\;\;\;\mbox{on}\;\;\bar{\Omega}.
 \end{cases}
 \end{equation}
 Then for any $0<\epsilon<T$ and $K\Subset\bar{\Omega}\setminus \{a_1,...,a_l\}$,
  there exists $C_1,C_2>0$ depending on
  $\Omega, T, C_0, N_1,...,N_l, a_1,..., a_l,\epsilon,K$ such that
 \begin{equation}
 |u(z,t)|+|\dot{u}(z,t)|+\Delta u(z,t)\leq C,\;\;\forall (z,t)\in K\times(\epsilon,T).
 \end{equation}
 \end{The}
 By Lemma \ref{utminusu0.lem.weak sec}, we have
 $$\sup\varphi\geq u(z,t)\geq u_0(z)-C(t)\geq \sum\limits_{j=1}^lN_j\log |z-a_j|-C_0-C(t).$$
 By Lemma \ref{dotu.lem}, we also have the bounds of $\dot{u}$. Then it remains to estimate  $\Delta u$. 
 \subsection{Bounds on $\mathbf{\nabla u}$}
 \begin{flushleft}
Let $0<\epsilon<T$. We need to estimate $\nabla u$ near $\partial\Omega\times (\epsilon, T)$ in order to bound 
 $\Delta u$ on $\partial\Omega\times (\epsilon, T)$.
 \end{flushleft}
 \begin{Lem}\label{gradient boundary. lem}
 For any $T>\epsilon>0$, there exists $B>0$ depending only on $n,\Omega, T, C_0, a_1,...,a_l$, $N_1,..., N_l,
  \epsilon$ such that
 $$\sup_{\partial\Omega\times (\epsilon,T)}|\nabla u|\leq B.$$
 \end{Lem}
 \begin{proof}
 For any $(z,t)\in\partial\Omega\times (0,T), \xi\in T_{\R,z}\partial\Omega$, we have 
 $$|u_{\xi}(z,t)|=|\varphi_{\xi}(z,t)|\leq C_0.$$
 It remains to show that , for any $(z,t)\in \partial\Omega\times (\epsilon,T)$,
 $$|u_{\eta}(z,t)|\leq B,$$ 
 where $\eta$ is an interior normal vector of $\partial\Omega$ at $z$, $\|\eta\|=1$.\\
 We need to construct functions $\U , h\in C^{\infty}(\bar{\Omega}\times [0,T))$ such that
 \begin{center}
 $\U\leq u\leq h$ on $\bar{\Omega}\times [0,T)$,\\
 $\U=\varphi=h$ on $\partial\Omega\times (\epsilon,T),$ 
 \end{center}
 and $|\nabla\U|, |\nabla h|$ are bounded by a constant which 
 depending only on $n$,$\Omega$, $T$, $C_0$,$ a_1$,...,$a_l$, $N_1$,...,$ N_l$, $ \epsilon$.\\
 Let $\U_0$ be a smooth plurisubharmonic function on $\bar{\Omega}$ such that $\U_0\leq u_0$ and
 \begin{center}
 $\U_0=\sum N_j\log |z-a_j|-(1+\epsilon)C_0\;\;$ near $\partial\Omega$.
 \end{center}
 Let $h_0$ be a harmonic function on $\bar{\Omega}$ such that
 \begin{center}
 $h_0=\sum N_j\log |z-a_j|-(1+\epsilon)C_0\;\;$ on $\partial\Omega$.
 \end{center}
 Let $\zeta: \R\to [0,1]$ be a smooth increasing function such that $\zeta (0)=0$, $\zeta(\epsilon)=1$.\\
 We consider the functions\\
 \begin{center}
 $\underline{\varphi}(z,t)=(1-\zeta(t))h_0(z)+\zeta (t) \varphi(z,t),$\\
 $\U(z,t)=\U_0(z)+\underline{\varphi}(z,t)-h_0(z)+B_1\rho (z),$
 \end{center}
 where 
 $\rho\in C^{\infty}(\bar{\Omega})$ such that $dd^c\rho\geq dd^c|z|^2$ and $\rho|_{\partial\Omega}=0$
  and $B_1\geq\frac{1}{n}\exp(\sup\dot{\underline{\varphi}}+3A\sup |\underline{\varphi}|+\sup |g|)$
  such that $\U (., t)\in SPSH(\bar{\Omega})$ for any $t\in [0,T]$. \\
 We have
 \begin{flushleft}
 $\U(z,0)=\U_0(z)+B_1\rho (z)\leq u_0(z)$,\\
 $\U|_{\partial\Omega\times (0,T)}=\underline{\varphi}|_{\partial\Omega\times (0,T)}
 \leq \varphi|_{\partial\Omega\times (0,T)},$\\
 $
 \dot{\U}-\log\det\U_{\alpha\bar{\beta}}+A\U-g(z,t)\leq \dot{\underline{\varphi}}-n\log B_1
 +A(\sup \U_0+2\sup|\underline{\varphi}|)+\sup |g|
 \leq 0.
 $
 \end{flushleft}
 It follows from Corollary \ref{compa log} that $\U\leq u$.
 
 Now, let $h:\bar{\Omega}\times [0,T)\to \R$ be a spatial harmonic function satisfying
 $$h|_{\partial\Omega\times[0,T)}=\varphi|_{\partial\Omega\times [0,T)}.$$
 We have
 $$\begin{cases}
 \U\leq u\leq h \mbox{ on } \bar{\Omega}\times [0,T),\\
  \U=\varphi=h \mbox{ on } \partial\Omega\times (\epsilon,T).
 \end{cases}$$
 Then for any $(z,t)\in \partial\Omega\times(\epsilon,T)$, we have
 $$\U_{\eta}(z,t)\leq u_{\eta}(z,t)\leq h_{\eta}(z,t),$$
 where $\eta$ is an interior normal vector of $\partial\Omega$ at $z$, $\|\eta\|=1$.\\
 Hence
 $$|u_{\eta}(z,t)|\leq \sup\{|\U_{\eta}(z,t)|,|h_{\eta}(z,t)|\}\leq B_2,$$
 where $B_2>0$ depends only on $n$,$\Omega$, $T$, $C_0$,$ a_1$,...,$a_l$, $N_1$,...,$ N_l$, $ \epsilon$.\\
 \end{proof}
 \begin{Lem}\label{gradient interior.lem}
 For any $T> 2\epsilon>0$ and $K\Subset\bar{\Omega}\setminus\{a_1,...,a_l\}$,
 there exists $B>0$ depending only on $n,\Omega, T, C_0, a_1,...,a_l$, $N_1,..., N_l,$
   $K,\epsilon$ such that
  $$\sup_{K\times (2\epsilon,T)}|\nabla u|\leq B.$$
 \end{Lem}
 \begin{proof}
 We will use the technique of Blocki as in \cite{Blo08}. 
 
 By Lemma \ref{dotu.lem}, there exists $M>0$ depending only on 
 $n,\Omega, T, C_0,\epsilon,$ $N_1,...,N_l$
 such that
 \begin{equation}\label{dotu.eq.bolocki}
 \begin{cases}
 \dot{u}\leq M(\sum\log_{-}|z-a_j|+1)\;\; \mbox{on}\;\bar{\Omega}\times [\epsilon,T),\\
 \dot{u}+Au-g\leq M(\sum\log_{-}|z-a_j|+1)\;\;\mbox{on}\; \bar{\Omega}\times [\epsilon,T),
 \end{cases}
 \end{equation}
 where $\log_{-}|z-a_j|=\max\{-\log|z-a_j|, 0\}$.
 
 Let $f_1,...,f_N \in \{f\in \mathcal{O}(\C^n,C): \log |f|\leq \sum\limits_{j=1}^lM\log |z-a_j|+O(1)\}$ satisfy
 $$\prod_{j=1}^l|z-a_j|^{2M}=\sum_{j=1}^N |f_j|^2.$$
 The $f_j$ are in fact polynomials, which could be written explicitly, but tediously, using
 the multinomial formula. 
 Note that the choice of $f_1,...,f_N$ depends only on $a_1,...,a_l, M$. Then there exists $C_1>0$
 depending only on $\Omega, a_1,...,a_l,M$ such that
 \begin{equation}\label{fjestimate.eq}
 \dfrac{\sum\limits_{j=1}^N|(f_j)_{\alpha}|.|f_j|}{\sum\limits_{j=1}^N|f_j|^2}
 \leq\sum\limits_{j=1}^l\dfrac{C_1}{|z-a_j|} \mbox{  on  }\Omega.
 \end{equation}
 Without loss of generality, we can assume that $0\in\Omega$ and $C_0>1$.
  We denote, for $(z,t)\in \bar{\Omega}\times [\epsilon,T)$,
 \begin{center}
 $k(t)=(n+1)\log (t-\epsilon),$\\
 $\gamma (u)=\log (3C_0-u)-\log (2C_0-u),$\\
 $\psi_j=f_j\nabla u, \;\forall j=1,...,N,$\\
 $\psi_0=\sum_{j=1}^{n}|\psi_j|^2,$\\
 $\psi=\log\psi_0+\gamma(u)+k(t)+\eta |z|^2,$
 \end{center}
 where $\eta=\frac{1}{4(diam (\Omega)+1)^2}$.
 
 We will show that
 $$\sup_{\bar{\Omega}\times [\epsilon,T']}\psi\leq \tilde{B},\;\forall T'\in (\epsilon,T),$$
 where $\tilde{B}$ depends only on $n,\Omega, T, C_0, a_1,...,a_l$, $N_1,..., N_l, K,\epsilon$.
 
 Let $(z_0,t_0)\in\bar{\Omega}\times [\epsilon,T']$ satisfy
 $$\psi (z_0,t_0)=\sup_{\bar{\Omega}\times [\epsilon,T']}\psi.$$
  By an orthogonal change of coordinates, we can assume that $(u_{\alpha\bar{\beta}}(z_0,t_0))$ is diagonal.
   For convenience, we denote $\lambda_{\alpha}=u_{\alpha\bar{\alpha}}(z_0,t_0)$.
  
  Assume that 
  \begin{equation}\label{psigeqb1.eq.blocki}
  \psi (z_0,t_0)\geq B_1+1,
 \end{equation}
 where $B_1\geq\sup_{\partial\Omega\times (\epsilon,T)}\psi$ is a constant depending only on 
 $n,\Omega, T, C_0$, $a_1,...,a_l$, $N_1,..., N_l$, $K,\epsilon$.
 
 Then $z_0\in\Omega\setminus\{a_1,...,a_l\}$,$t_0\in (\epsilon,T']$. Hence, $\dot{\psi}(z_0,t_0)\geq 0$,
 $\nabla\psi (z_0,t_0)=0$ and $(\psi_{\alpha\bar{\beta}}(z_0,t_0))$ is negative. We have, at $(z_0,t_0)$,
 \begin{equation}\label{blocki.eq.derivatives}
 \begin{cases}
 \dfrac{(\psi_0)_{\alpha}}{\psi_0}=-\gamma'(u) u_{\alpha}-\eta\bar{z}_{\alpha},\\
 L(\psi):=\dot{\psi}-\sum u^{\alpha\bar{\beta}}\psi_{\alpha\bar{\beta}}\geq 0,
 \end{cases}
 \end{equation}
 where $(u^{\alpha\bar{\beta}})$ is the transpose of 
  inverse matrix of Hessian matrix  $(u_{\alpha\bar{\beta}})$.\\
 We compute, at $(z_0,t_0)$,
 \begin{flushleft}
 $\begin{array}{ll}
 L(\gamma(u)+k(t)+\eta |z|^2)&=\gamma'(u)\dot{u}+k'(t)
 -\gamma'(u)\sum u^{\alpha\bar{\beta}}u_{\alpha\bar{\beta}}
 -\gamma''(u)\sum u^{\alpha\bar{\beta}}u_{\alpha}u_{\bar{\beta}}-\eta\sum u^{\alpha\bar{\alpha}}\\
 &=\gamma'(u)(\dot{u}-n)+k'(t)-\gamma''(u)\sum\dfrac{|u_{\alpha}|^2}{\lambda_{\alpha}}
 -\eta\sum\dfrac{1}{\lambda_{\alpha}}. 
 \end{array}$
 $\begin{array}{ll}
 L(\log\psi_0)&=\dfrac{\dot{\psi}_0}{\psi_0}-\sum \dfrac{(\psi_0)_{\alpha\bar{\alpha}}}{\lambda_{\alpha}\psi_0}
 +\sum \dfrac{|(\psi_0)_{\alpha}|^2}{\lambda_{\alpha}\psi_0^2}\\
 &=\dfrac{\dot{\psi}_0}{\psi_0}-\sum\limits_{j,\alpha}
 \dfrac{|(\psi_j)_{\alpha}|^2+|(\psi_j)_{\bar{\alpha}}|^2+
 2Re(\langle(\psi_j)_{\alpha\bar{\alpha}},\psi_j\rangle)}{\lambda_{\alpha}\psi_0}
 +\sum \dfrac{|(\psi_0)_{\alpha}|^2}{\lambda_{\alpha}\psi_0^2}\\
 &\leq \dfrac{\dot{\psi}_0}{\psi_0}
 -\sum\limits_{j,\alpha} \dfrac{2Re(\langle(\psi_j)_{\alpha\bar{\alpha}},\psi_j\rangle)}{\lambda_{\alpha}\psi_0}
  +\sum \dfrac{|(\psi_0)_{\alpha}|^2}{\lambda_{\alpha}\psi_0^2}\\
 &=\dfrac{\dot{\psi}_0}{\psi_0}-\sum\limits_{j,\alpha} 
 \dfrac{2Re(\lambda_{\alpha}(f_j)_{\alpha}\bar{f}_ju_{\bar{\alpha}})}{\lambda_{\alpha}\psi_0}
 -\sum\limits_{j,\alpha}
 \dfrac{2Re(\langle f_j \nabla u_{\alpha\bar{\alpha}},\psi_j\rangle)}{\lambda_{\alpha}\psi_0}
 +\sum \dfrac{|(\psi_0)_{\alpha}|^2}{\lambda_{\alpha}\psi_0^2}\\
 &=\dfrac{2Re\langle\nabla\dot{u},\nabla u \rangle}{|\nabla u|^2}
 -\sum\limits_{j,\alpha} 
  \dfrac{2Re((f_j)_{\alpha}\bar{f}_ju_{\bar{\alpha}})}{\psi_0}
  -\sum\limits_{\alpha}
   \dfrac{2Re(\langle \nabla u_{\alpha\bar{\alpha}},\nabla u\rangle)}{\lambda_{\alpha}|\nabla u|^2}
   +\sum \dfrac{|(\psi_0)_{\alpha}|^2}{\lambda_{\alpha}\psi_0^2}\\
   &=\dfrac{2Re(\langle L(\nabla u), \nabla u \rangle)}{|\nabla u|^2}
   -\sum\limits_{j,\alpha} \dfrac{2Re((f_j)_{\alpha}\bar{f}_ju_{\bar{\alpha}})}{\psi_0}
   +\sum \dfrac{|(\psi_0)_{\alpha}|^2}{\lambda_{\alpha}\psi_0^2}\\
   &=\dfrac{2Re(\langle -A\nabla u +\nabla g, \nabla u \rangle)}{|\nabla u|^2}
      -\sum\limits_{j,\alpha} \dfrac{2Re((f_j)_{\alpha}\bar{f}_ju_{\bar{\alpha}})}{\psi_0}
      +\sum \dfrac{|(\psi_0)_{\alpha}|^2}{\lambda_{\alpha}\psi_0^2}\\
      &\leq \dfrac{2|\nabla g|}{|\nabla u|}+\sum\limits_{j=1}^l\dfrac{2nC_1}{|z-a_j||\nabla u|}
      +\sum \dfrac{|(\psi_0)_{\alpha}|^2}{\lambda_{\alpha}\psi_0^2}
 \end{array}$
 \end{flushleft}
 Then
 \begin{equation}\label{Lpsi before change gra.eq.blocki}
 \begin{array}{ll}
 L(\psi)&\leq \dfrac{2|\nabla g|}{|\nabla u|}+\sum\limits_{j=1}^l\dfrac{2nC_1}{|z-a_j||\nabla u|}
  +\sum \dfrac{|(\psi_0)_{\alpha}|^2}{\lambda_{\alpha}\psi_0^2}+\gamma'(u)(\dot{u}-n)\\
  &+k'(t_0)-\gamma''(u)\sum\dfrac{|u_{\alpha}|^2}{\lambda_{\alpha}}
        -\eta\sum\dfrac{1}{\lambda_{\alpha}}.
       \end{array} 
   \end{equation}  
   By \eqref{blocki.eq.derivatives}, we have,
   \begin{equation}\label{changegradient.eq.blocki}
   \sum \dfrac{|(\psi_0)_{\alpha}|^2}{\lambda_{\alpha}\psi_0^2}
   \leq 2\sum \dfrac{(\gamma'(u))^2|u_{\alpha}|^2+\eta^2|z_{\alpha}|^2}{\lambda_{\alpha}}
   \leq 2(\gamma'(u))^2\sum\dfrac{|u_{\alpha}|^2}{\lambda_{\alpha}}
   +\dfrac{\eta}{2}\sum\dfrac{1}{\lambda_{\alpha}}.
   \end{equation}
   Note that
   \begin{flushleft}
   $\gamma'(u)=\dfrac{1}{2C_0-u}-\dfrac{1}{3C_0-u}=\dfrac{C_0}{(2C_0-u)(3C_0-u)},$\\
   $\gamma''(u)=\dfrac{1}{(2C_0-u)^2}-\dfrac{1}{(3C_0-u)^2}
   =\dfrac{C_0(5C_0-2u)}{(2C_0-u)^2(3C_0-u)^2},$\\
   $\gamma''(u)-2(\gamma'(u))^2\geq \dfrac{C_0^2}{(3C_0-u)^4}.$
\end{flushleft} 
Then, we have, by \eqref{Lpsi before change gra.eq.blocki}, \eqref{changegradient.eq.blocki},
\begin{flushleft}
$\begin{array}{ll}
 L(\psi)&\leq \dfrac{2|\nabla g|}{|\nabla u|}+\sum\limits_{j=1}^l\dfrac{2nC_1}{|z-a_j||\nabla u|}
  +\gamma'(u)(\dot{u}-n)\\
  &+k'(t_0)-\dfrac{C_0^2}{(3C_0-u)^4}\sum\dfrac{|u_{\alpha}|^2}{\lambda_{\alpha}}
        -\dfrac{\eta}{2}\sum\dfrac{1}{\lambda_{\alpha}}.
       \end{array}$
\end{flushleft}
By \eqref{dotu.eq.bolocki}, \eqref{psigeqb1.eq.blocki}, there exists $C_2>0$ depending only on
$n,\Omega, T, C_0$, $a_1,...,a_l$, $N_1,..., N_l$, $K,\epsilon$ such that
\begin{center}
$
 L(\psi)\leq C_2-M\sum\log |z-a_j|
  +k'(t_0)-\dfrac{C_0^2}{(3C_0-u)^4}\sum\dfrac{|u_{\alpha}|^2}{\lambda_{\alpha}}
        -\dfrac{\eta}{2}\sum\dfrac{1}{\lambda_{\alpha}}.
      $
\end{center}
We can also assume that $C_2-M\sum\log |z-a_j|>0$.

By the condition $L(\psi)|_{(z_0,t_0)}\geq 0$, we have,
\begin{equation}\label{lambdalambda.eq.blocki}
\dfrac{C_0^2}{(3C_0-u)^4}\sum\dfrac{|u_{\alpha}|^2}{\lambda_{\alpha}}
 +\dfrac{\eta}{2}\sum\dfrac{1}{\lambda_{\alpha}}
  \leq C_2-M\sum\log |z-a_j|+k'(t_0).
\end{equation}
Then
\begin{flushleft}
$\begin{array}{ll}
\dfrac{1}{\lambda_{\alpha}}
=\dfrac{\prod\limits_{\beta\neq\alpha}\lambda_{\beta}}{\prod\limits_{\beta=1}^n\lambda_{\beta}}
&=(\prod\limits_{\beta\neq\alpha}\lambda_{\alpha})e^{-\dot{u}-Au+g}\\
&\geq \left(\dfrac{\eta}{2}\right)^{n-1}(C_2-M\sum\log |z-a_j|+k'(t_0))^{-n+1}e^{-\dot{u}-Au+g}.
\end{array}$
\end{flushleft}
Hence, by \eqref{lambdalambda.eq.blocki},
\begin{flushleft}
$\begin{array}{ll}
|\nabla u|^2 &\leq \dfrac{(3C_0-u)^4}{C_0^2}\left(\dfrac{2}{\eta}\right)^{n-1}
(C_2-M\sum\log |z-a_j|+k'(t_0))^ne^{\dot{u}+Au-g}\\
&\leq C_2\dfrac{(3C_0-u)^4}{C_0^2}e^{\dot{u}+Au-g}(1+C_2-M\sum\log |z-a_j|)^n
(1+k'(t_0))^n\\
&\leq C_3\dfrac{(3C_0-u)^4}{C_0^2}e^{\dot{u}+Au-g}
(1+k'(t_0))^n\prod |z-a_j|^{-M/2},
\end{array}$
\end{flushleft}
where $C_3>0$ depends only on $n,\Omega, T, C_0$, $a_1,...,a_l$, $N_1,..., N_l$, $K,\epsilon$.

Then, by \eqref{dotu.eq.bolocki} and Lemma \ref{utminusu0.lem.weak sec}, we have,
\begin{equation}
|\nabla u|^2\leq C_4\dfrac{(3C_0+C(t_0)-u_0)^4}{C_0^2}\left(1+k'(t_0)\right)^n\prod |z-a_j|^{-3M/2},
\end{equation}
where $C(t_0)$ is defined as in Lemma \ref{utminusu0.lem.weak sec} and $C_4>0$ 
depends only on $C_3, M, a_1,...,a_l$.

Note that $u_0\geq\sum N_j\log |z-a_j|-C_0$. Then
\begin{equation}\label{1.eq.blocki}
|\nabla u|^2 \leq C_5(1+k'(t_0))^n\prod |z-a_j|^{-2M},
\end{equation}
where $C_5>0$ depends only on $n,\Omega, T, C_0$, $a_1,...,a_l$, $N_1,..., N_l$, $K,\epsilon$.

By \eqref{psigeqb1.eq.blocki} and \eqref{1.eq.blocki}, we have,
\begin{equation}
\dfrac{e^{B_1-2}}{(t_0-\epsilon)^{n+1}}
=e^{B_1-2-k(t_0)}
\leq \psi_0(z_0,t_0)
\leq C_5 (k'(t_0)+1)^n=C_5\left(\dfrac{t_0+n+1-\epsilon}{t_0-\epsilon}\right)^n.
\end{equation}
Hence, $t_0\geq t_1>\epsilon$, where $t_1$ depends only on 
$n,\Omega, T, C_0$, $a_1,...,a_l$, $N_1,..., N_l$, $K,\epsilon$. We have, by \eqref{1.eq.blocki},
\begin{equation}\label{2.eq.blocki}
\psi (z_0,t_0)\leq \tilde{B},
\end{equation}
where $\tilde{B}>B_1+1>0$ depends only on 
$n,\Omega, T, C_0$, $a_1,...,a_l$, $N_1,..., N_l$, $K,\epsilon$.
 
 Note that $\tilde{B}$ is independent of $T'$. Then,
 $$\sup_{\bar{\Omega}\times [\epsilon,T)}\psi\leq \tilde{B}.$$
 In particular, there exists $B>0$ depending only on 
$n,\Omega, T, C_0$, $a_1,...,a_l$, $N_1,..., N_l$, $K,\epsilon$ such that
$$\sup_{K\times (2\epsilon,T)}|\nabla u|\leq B.$$
 \end{proof}
 \subsection{Higher order estimates}
 \begin{Lem}\label{2 order boundary}
  For any $T>\epsilon>0$, there exists $B>0$ depending only on $n,\Omega, T, C_0, a_1,...,a_l$, $N_1,..., N_l,
   \epsilon$ such that
  $$\sup_{\partial\Omega\times (\epsilon,T)}|D^2 u|\leq B.$$
 \end{Lem}
 \begin{proof}
    By Lemma \ref{dotu.lem} and Lemma \ref{gradient interior.lem}, we can estimate $\dot{u}$ 
    and $\nabla u$ near $\partial\Omega\times (\epsilon,T)$. Then the proof of this lemma is the same as
    the case $u_0\in L^{\infty}(\Omega)$ (see \cite{Do15}). 
    \end{proof}
    Using the 2-order estimates on $\partial\Omega\times (0,\epsilon)$, we will estimate $\Delta u$ 
    on $K\times (2\epsilon, T)$, for any $K\Subset\bar{\Omega}\setminus\{a_1,...,a_l\}$.
 \begin{Lem}
 For any $T> 2\epsilon>0$ and $K\Subset\bar{\Omega}\setminus\{a_1,...,a_l\}$,
  there exists $B>0$ depending only on $n,\Omega, T, C_0, a_1,...,a_l$, $N_1,..., N_l,$
    $K,\epsilon$ such that
   $$\sup_{K\times (2\epsilon,T)}\Delta u\leq B.$$
 \end{Lem}
 \begin{proof}
 By Lemma \ref{utminusu0.lem.weak sec} and Lemma \ref{2 order boundary},
  there exist $B_1, B_2>0$ depending only on 
  $n,\Omega, T, C_0, a_1,...,a_l$, $N_1,..., N_l,  \epsilon$ such that
  \begin{equation}
  \sup_{\Omega\times (\epsilon,T)} (-u+u_{\epsilon}+|z|^2)\leq B_1,
  \end{equation}
  \begin{equation}\label{lap bound.eq.lap in}
  B_1+\sup_{\partial\Omega\times (\epsilon,T)}(t\log \Delta u -u+u_{\epsilon}+|z|^2) \leq B_2.
  \end{equation}
  where $u_{\epsilon}=u(.,\epsilon)$.
  
 We consider the function $\phi\in C^{\infty}(\bar{\Omega}\times [\epsilon,T))$ defined by
 $$\phi =t\log \Delta u -u+u_{\epsilon}-B_3 (t-\epsilon)+|z|^2,$$
 where $B_3=C_0(A+T+1)+\log (n!)+n+1$.

 We will  show that
  \begin{equation}\label{lap in.eq.lap in}
  \sup_{\bar{\Omega}\times (\epsilon,T)}\phi \leq B_2.
  \end{equation}
  Indeed, if there exists $(z_0,t_0)\in\bar{\Omega}\times[\epsilon,T)$ satisfying
  $$\phi (z_0,t_0)>B_2,$$
  then $z_0\in \Omega$, $t_0>\epsilon$. Without loss of generality, we can assume that
  $$\phi (z_0,t_0)=\sup_{\bar{\Omega}\times (\epsilon,t_0)}\phi.$$
  Denote by $L$ the operator
   $$L(\psi)=\dot{\psi}-\sum u^{\alpha\bar{\beta}}\psi_{\alpha\bar{\beta}},\;\;\forall
   \psi\in C^{\infty}(\bar{\Omega}\times [\epsilon,T)),$$
   where $(u^{\alpha\bar{\beta}})$ is the transpose of 
   inverse matrix of Hessian matrix  $(u_{\alpha\bar{\beta}})$.
   
   We have
   $$L(\phi)|_{(z_0,t_0)}\geq 0.$$
   We compute
   \begin{flushleft}
   $\begin{array}{ll}
   L(\phi)&=t\dfrac{\Delta\dot{u}}{\Delta u}+\log\Delta u-\dot{u}-B_3
   -t\sum u^{\alpha\bar{\beta}}(\log\Delta u)_{\alpha\bar{\beta}}
   +\sum u^{\alpha\bar{\beta}}(u_{\alpha\bar{\beta}}
   -(u_{\epsilon})_{\alpha\bar{\beta}})-\sum u^{\alpha\bar{\alpha}}\\[8pt]
   &\leq t\dfrac{\Delta\dot{u}}{\Delta u}+\log\Delta u-\dot{u}-B_3
      -t\sum u^{\alpha\bar{\beta}}(\log\Delta u)_{\alpha\bar{\beta}}+n-\sum u^{\alpha\bar{\alpha}}.
   \end{array}$
   \end{flushleft}
   Applying Theorem \ref{lap 2}, we have
   \begin{flushleft}
      $\begin{array}{ll}
      L(\phi) &\leq t\dfrac{\Delta\dot{u}}{\Delta u}+\log\Delta u-\dot{u}-B_3
      -t\dfrac{\Delta \log\det (u_{\alpha\bar{\beta}})}{\Delta u}+n-\sum u^{\alpha\bar{\alpha}}\\[8pt]
      &=t\dfrac{\Delta(\dot{u}-\log\det (u_{\alpha\bar{\beta}}))}{\Delta u}
      +\log\Delta u-\dot{u}-B_3+n-\sum u^{\alpha\bar{\alpha}}\\[8pt]
      &=t\dfrac{-A\Delta u+\Delta g}{\Delta u}
            +\log\Delta u-\dot{u}-B_3+n-\sum u^{\alpha\bar{\alpha}}\\[8pt]
      &\leq t\dfrac{\Delta g}{\Delta u}
                  +\log\Delta u-\dot{u}-B_3+n-\sum u^{\alpha\bar{\alpha}}.
      \end{array}$
         \end{flushleft}
Applying Theorem \ref{lap 1}, we have
$$\log\Delta u\leq \log n+\log\det (u_{\alpha\bar{\beta}})+(n-1)\log(\sum u^{\alpha\bar{\alpha}}).$$
Then
\begin{flushleft}
$\begin{array}{ll}
 L(\phi) &\leq t\dfrac{\Delta g}{\Delta u} +\log n+\log\det (u_{\alpha\bar{\beta}})+
 (n-1)\log(\sum u^{\alpha\bar{\alpha}})-\dot{u}-B_3+n-\sum u^{\alpha\bar{\alpha}}\\[8pt]
 &\leq t\dfrac{\Delta g}{\Delta u} +\log n-B_3+n+Au-g+(n-1)\log(\sum u^{\alpha\bar{\alpha}})
 -\sum u^{\alpha\bar{\alpha}}\\[8pt]
 &\leq t\dfrac{\Delta g}{\Delta u} +\log n-B_3+n+(A+1)C_0+\log ((n-1)!)\\[8pt]
 &=t\dfrac{\Delta g}{\Delta u} +\log (n!)-B_3+n+(A+1)C_0,
\end{array}$
\end{flushleft}
where third inequality holds due to the conditions \eqref{condition.g phi.a priori}, \eqref{condition.u0.a priori}
and the inequalities
\begin{center}
 $ (n-1)\log x-x \leq (n-1)\log(n-1)-(n-1), \;\; \forall x>0, $ 
\end{center}
and
\begin{center}
 $\dfrac{(n-1)^{n-1}}{(n-1)!}  \leq \exp \left(n-1\right).$
\end{center}
Hence 
$$L(\phi)|_{(z_0,t_0)}\leq C_0T+\log (n!)-B_3+n+(A+1)C_0<0.$$
We have a contradiction. Thus 
$$\sup_{\bar{\Omega}\times (\epsilon,T)}\phi \leq B_2.$$
In particular, there exists $B>0$ depending only on $n,\Omega, T, C_0, a_1,...,a_l$, $N_1,..., N_l,$
    $K,\epsilon$ such that
   $$\sup_{K\times (2\epsilon,T)}\Delta u\leq B.$$
  
 \end{proof}
 \section{Proof of Theorem \ref{main.description}}
 \subsection{Smoothness}
 As in the proof of Proposition \ref{exist unique.prop.sec weak}, we can construct a sequence of functions
 $u_m\in C^{\infty}(\bar{\Omega}\times [0,T))$ satisfying
 \begin{equation}
    \begin{cases}
    \begin{array}{ll}
    u_m(.,t)\in SPSH(\Omega),\\
    \dot{u}_m=\log\det (u_m)_{\alpha\bar{\beta}}-Au_m+f(z,t)\;\;\;&\mbox{on}\;\Omega\times (0,T),\\
    u_m=\varphi_m\rightrightarrows\varphi&\mbox{on}\;\partial\Omega\times [0,T),\\
    u_m=\varphi &\mbox{on}\;\partial\Omega\times (\epsilon_m,T),\\
    u_m=u_{0,m}\searrow u_0&\mbox{on}\;\bar{\Omega}\times\{ 0\},\\
    u=\lim\limits_{m\rightarrow \infty} u_m.
    \end{array}
    \end{cases}
    \end{equation}
    where $\epsilon_m\searrow 0$.
    
    Applying Theorem \ref{the.a priori}, we obtain, for any $K\Subset \bar{\Omega}\setminus\{a_1,...,a_l\}$
    and $0<\epsilon<T$,
    \begin{center}
    $\|u_m\|_{C^2(K\times[\epsilon,T))}\leq C_{K,\epsilon},$
    \end{center}
     where $C_{K,\epsilon}$ depends only on $n, \Omega, T, C_0, N_1,...,N_l, a_1,..., a_l,\epsilon,K$.
     
    It follows from $C^{2,\alpha}$-estimates (see, for example, \cite{Do15}) that
    \begin{center}
    $\|u_m\|_{C^{2,\gamma}(K\times[\epsilon,T))}\leq C_{K,\epsilon,\gamma},$
    \end{center}
    where $\gamma$ and $C_{K,\epsilon,\gamma}$ depend only on
     $n, \Omega, T, C_0, N_1,...,N_l, a_1,..., a_l,\epsilon,K$.
     
   By Ascoli theorem, we obtain
   \begin{equation}\label{Ascoli.eq.proofmain2}
   u_m\stackrel{C^{2,\gamma/2}(K\times[\epsilon,T'])}{\longrightarrow} u,\;\mbox{ as} \; m\rightarrow\infty,
   \end{equation}
   where $\epsilon<T'<T$.
   
   Then $u\in C^{2,\gamma/2}(K\times[\epsilon,T'] )$.  Applying regularity theorem (see, for example,
   \cite{Do15}), we have
   $u\in C^{\infty}(K\times (\epsilon,T'))$. Hence, $u\in C^{\infty}(K\times (0,T)).$
   
   Fix $0<r<\min\limits_{j\neq k}|a_j-a_k|$. We need to show that 
   $u\in C^{\infty}(B_r(a_j)\times (\epsilon_A(N_j),T))$ when $T>\epsilon_A(N_j)$.
   
   Fix $\epsilon>0$. If $A=0$, we have, for $t=\epsilon_0(N_j)+2\epsilon=\frac{N_j}{2n}+\epsilon$,
   \begin{center}
   $(dd^cu_m)^n=e^{\dot{u}_m(z,\epsilon_0(N_j)+2\epsilon)-f(z,\epsilon_0(N_j)+2\epsilon)}dV=d\mu.$
   \end{center}
   Applying Lemma \ref{dotu.lem} and Proposition \ref{singular.prop.weak sec}, we have, for $m\gg 1$,
   \begin{flushleft}
   $\begin{array}{ll}
   \exp\left(\dot{u}_m(z,\epsilon_0(N_j)+2\epsilon)-f(z,\epsilon_0(N_j)+2\epsilon)\right)
   &\leq C_1\exp\left(\frac{u_m(z,\epsilon_0+2\epsilon)-u_m(z,\epsilon)}{\epsilon_0(N_j)+\epsilon }\right)\\[6pt]
   &\leq C_2\exp\left(-\frac{u_m(z,\epsilon)}{\epsilon_0(N_j)+\epsilon} \right)\\[6pt]
   &\leq C_3\exp\left(-\frac{u_m(z,0)}{\epsilon_0(N_j)+\epsilon} \right)\\[6pt]
   &\leq C_3\exp\left(-\frac{u_0(z)}{\epsilon_0(N_j)+\epsilon} \right)\\[6pt]
   &=C_3\exp\left(-\frac{2n u_0(z)}{N_j+2n\epsilon} \right)
    \end{array}$
   \end{flushleft}
   where $C_1, C_2, C_3>0$ depend only on $\Omega, \epsilon, T, f, \varphi$.
   
   Then, we have, for $K\Subset B_r(a_j)$,
   \begin{flushleft}
   $\begin{array}{ll}
   \mu (K)
   &\leq C_3\int_K \exp\left(-\frac{2n u_0(z)}{N_j+2n\epsilon} \right) dV\\[8pt]
   &\leq C_4\int_K \exp\left(-\frac{2nN_j\log |z-a_j|}{N_j+2n\epsilon}\right) dV\\[8pt]
   &=C_4\int_K |z-a_j|^{-\frac{2nN_j}{N_j+2n\epsilon}}dV\\[8pt]
   &\leq C_5 (\int_K|z-a_j|^{-\frac{2nN_j}{N_j+n\epsilon}}dV)^{\frac{N_j+n\epsilon}{N_j+2n\epsilon}}
   (\int_KdV)^{\frac{n\epsilon}{N_j+2n\epsilon}}\\[8pt]
   &\leq C_6 (\int_KdV)^{\frac{n\epsilon}{N_j+2n\epsilon}}\\[8pt]
   &\leq C_7 (Cap(K))^2.
   \end{array}$
   \end{flushleft}
   where $C_4, C_5, C_6, C_7>0$ are independent of $m$. The last inequality holds due to \cite{AT84} and  \cite{Ze01}.
   
   Applying Kolodziej theorem (Theorem B \cite{Kol98}), we have 
   \begin{center}
   $\|u_m(.,\epsilon_A(N_j)+2\epsilon)\|_{L^{\infty}(B_r(a_j))}\leq C_8$,
   \end{center}
   where $C_8$ is independent of $m$. Then $\|u(.,\epsilon_A(N_j)+2\epsilon)\|_{L^{\infty}(B_r(a_j))}\leq C_8$.
   Applying the case $"u_0\in L^{\infty}(\Omega)"$ (see \cite{Do15}) , we have 
   $u\in C^{\infty}(B_r(a_j)\times (\epsilon_A(N_j)+2\epsilon,T))$.
   
   Let $\epsilon\to 0$. We have $u\in C^{\infty}(B_r(a_j)\times (\epsilon_A(N_j),T))$.
   
   If $A>0$, by the same arguments, we also obtain 
   $u\in C^{\infty}(B_r(a_j)\times (\epsilon_A(N_j),T))$.
   
   Thus $u\in C^{\infty}(Q)$.
   
   By \eqref{Ascoli.eq.proofmain2}, for any 
   $(z,t)\in \bar{\Omega}\setminus\{a_1,...,a_l\}\times (0,T)$,
   \begin{equation}\label{KRF.proofmain2}
   \dot{u}=\log\det (u_{\alpha\bar{\beta}})-Au+f(z,t).
   \end{equation}
   Taking the limits, we obtain \eqref{KRF.proofmain2} on $Q$.
  \subsection{Singularity}Let $j\in\{1,...,l\}$. 
 By Proposition \ref{singular.prop.weak sec}, we have $u(a_j,t)=-\infty$ for any $t\in [0,\epsilon_A(n_j))$.
 
 We need to show that if $n_j=N_j$ then $\nu_{u(., t)}(a_j)=k(N_j,t)$.
 
 By the proof of Proposition \ref{singular.prop.weak sec}, we have, for any $t\in [0,\epsilon_A(N_j))$,
 \begin{equation}\label{nugeq.proofmain2}
 \nu_{u(.,t)}(a_j)\geq k(N_j,t),
 \end{equation}
 Then, it remains to show that
 \begin{equation}\label{nuleq.proofmain2}
 \nu_{u(.,t)}(a_j)\leq k(N_j,t),
 \end{equation}
 for any $t\in [0,\epsilon_A(N_j))$.
 
 If there exist $t_0\in (0,\epsilon_A(N_j))$ and $\epsilon>0$  such that
 $$\nu_{u(.,t_0)}(a_j)\geq k(N_j,t_0)+\epsilon$$
 then, by Proposition \ref{singular.prop.weak sec}, there exists $t_1>\epsilon_A(N_j)$ such that
 $$\nu_{u(.,t_1)}(a_j)>0.$$
 This contradicts  the smoothness of $u$ on $Q$. Then we obtain \eqref{nuleq.proofmain2}.
 
 Combining \eqref{nugeq.proofmain2} and \eqref{nuleq.proofmain2}, we obtain
 $$\nu_{u(.,t)}(a_j)= k(N_j,t).$$
 \subsection{Continuity at zero}
 \begin{flushleft}
 \emph{a) Convergence in $L^1$.}
 \end{flushleft}
 Applying Lemma \ref{utminusu0.lem.weak sec}, we have
 \begin{equation}\label{liminf.conti0.proofmain2}
 \liminf\limits_{t\to 0}u(z,t)\geq u_0(z),
 \end{equation}
 for any $z\in\bar{\Omega}$.
 
  Note that $u$ is the limit of a decreasing sequence of smooth functions, then $u\in USC(\bar{\Omega}\times [0,T))$. We have
 \begin{equation}\label{limsup.conti0.proofmain2}
 \limsup\limits_{t\to 0}u(z,t)\leq u_0(z),
 \end{equation}
 for any $z\in\bar{\Omega}$.
 
 Combining \eqref{liminf.conti0.proofmain2} and \eqref{limsup.conti0.proofmain2}, we obtain
 $$\lim\limits_{t\to 0}u(z,t)=u_0(z).$$
 
 Hence, by the dominated convergence theorem, $u(., t)\rightarrow u_0$ in $L^1$, as $t\to 0$.
 
 \begin{flushleft}
 \emph{b) Convergence in capacity}
 \end{flushleft}
 Let $\epsilon>0$ and $\tilde{\Omega}$ be a neighbouhood of $\bar{\Omega}$. We need to show that there exists
 an open set $U_{\epsilon}$ such that 
 $$Cap_{\tilde{\Omega}}(\bar{\Omega}\setminus U_{\epsilon})\leq\epsilon$$
 and
 \begin{center}
 $u(.,t_{k_m})\rightrightarrows u_0$ on $U_{\epsilon}\cap \bar{\Omega}$.
 \end{center}
  Indeed, by quasicontinuity of plurisubharmonic function,
   there exists $K_{\epsilon}\Subset\Omega$ such that 
 $$Cap_{\tilde{\Omega}}(\bar{\Omega}\setminus Int (K_{\epsilon}))\leq\epsilon$$ 
 and $u_0$ is continuous on 
 $K_{\epsilon}$.
 
 By Lemma \ref{utminusu0.lem.weak sec} and Dini theorem, for any $t_k\searrow 0$, there exists
 a subsequence $t_{k_m}\searrow 0$ such that
 \begin{center} 
 $u(.,t_{k_m})\rightrightarrows u_0$ on $K_{\epsilon}$, as $m\to\infty$.
 \end{center}
 Then 
 \begin{center} 
  $u(.,t)\rightrightarrows u_0$ on $K_{\epsilon}$, as $t\to 0$.
  \end{center}
 
\end{document}